\documentclass[11pt]{article}
\usepackage[margin=1in]{geometry}
\usepackage[T1]{fontenc}
\usepackage{lmodern}
\usepackage{amsmath,amssymb,amsthm,mathtools,microtype}
\usepackage[colorlinks=true,linkcolor=black,citecolor=black,urlcolor=black]{hyperref}
\hypersetup{
 pdftitle={Uniform asymptotics and the Weil--Petersson constant},
 pdfauthor={Kefeng Liu and Hao Xu},
 pdfsubject={Large-genus asymptotics of Weil--Petersson volumes}}
\newtheorem{theorem}{Theorem}[section]
\newtheorem{lemma}[theorem]{Lemma}
\newtheorem{proposition}[theorem]{Proposition}
\newtheorem{corollary}[theorem]{Corollary}
\theoremstyle{remark}
\newtheorem{remark}[theorem]{Remark}
\numberwithin{equation}{section}
\newcommand{\WP}{\mathrm{WP}}
\newcommand{\dd}{\mathrm d}

\newcommand{\st}{\mathrm{stable}}
\newcommand{\bzero}{\mathbf 0}
\allowdisplaybreaks[1]
\title{Uniform asymptotics and the Weil--Petersson constant}
\author{Kefeng Liu \\ Chongqing University of Technology\\ kefengliu@cqut.edu.cn
\and 
Hao Xu \\ Chongqing University of Technology\\ haoxu@cqut.edu.cn}
\date{}
\begin{document}
\maketitle

\begin{abstract}
In their celebrated work on the large-genus asymptotic expansion of
Weil--Petersson volumes, Mirzakhani and Zograf left open the question
of whether the Weil--Petersson constant satisfies $C_{\WP}=\pi^{-1/2}$.
We resolve the conjecture affirmatively.
Our strategy is to consider a family of spectral curves interpolating
between $y=z-4z^3$ and $y=\sin(2\pi z)/(2\pi)$.
For each fixed number $n\geq2$ of marked points, we normalize the cusp
coefficients using the action integral and the derivative at
its first positive zero. We prove that the logarithmic derivatives
of these normalized coefficients with respect to the family parameter
are $O_n(g^{-1})$, uniformly along the family.
The proof combines the descendant recursion, positivity, the dilaton
equation, KdV, and coefficient asymptotics for the Painlev\'e I equation.
\end{abstract}
\setcounter{tocdepth}{1}
{\small\tableofcontents}

\section{Introduction}
Let $V^{\WP}_{g,n}$ denote the ordinary cusp Weil--Petersson volume
of genus $g$ with $n$ labelled markings. For each fixed $n\geq0$,
Mirzakhani and Zograf \cite{MZ} proved
\begin{equation}\label{eq:mz}
 V^{\WP}_{g,n}
 =C_{\WP}\frac{(4\pi^2)^{2g-3+n}\Gamma(2g-2+n)}{\sqrt g}
 \bigl(1+O_n(g^{-1})\bigr),
\end{equation}
where $C_{\WP}>0$ is independent of $n$.
The purpose of this paper is to determine $C_{\WP}$ by comparing
the WP volumes with the coefficients of a cubic spectral curve.
We use \eqref{eq:mz} to specify the conventional normalization.

Although many recursive formulas are available for intersection numbers on moduli spaces of curves, explicit computations in high genus remain challenging. 
Zograf \cite{Zog} developed an efficient algorithm for computing Weil--Petersson volumes in high genus.
He \cite{Zog} conjectured the expansion \eqref{eq:mz}, with
$C_{\WP}=\pi^{-1/2}$, on the basis of numerical calculations.
Mirzakhani \cite{Mir} proved initial large-genus estimates, and
Mirzakhani--Zograf \cite{MZ} established the expansion with an
undetermined constant. Building on these estimates, Mirzakhani pioneered the study of random
hyperbolic surfaces; see, for example, \cite{AM2,LW,WX} for the distribution of the first eigenvalue of a random hyperbolic surface in large genus. Related
large-genus problems for intersection numbers and their applications to quantum gravity and the study of Masur--Veech volumes
are considered in \cite{Agg,EGGGL,GPRS,GNYZ,GY,Hua}.

Eynard--Garcia-Failde--Gregori--Lewa\'nski--Schiappa
\cite{JT} derived the value
\(C_{\WP}=1/\sqrt{\pi}\) under a conjectural resurgent structure and the associated
identification of the relevant singularity data on the
Weil--Petersson generating series.

In our analysis, the constant \(C_{\WP}=1/\sqrt{\pi}\) follows from uniform
estimates for the descendant recursion and a coefficient
asymptotic for Painlev\'e I. We do not assume resurgence of the
Weil--Petersson series. For other connections between Painlev\'e
equations and intersection numbers, see \cite{GNYZ,IZ}.

We also note that the existence of an asymptotic expansion of the Weil--Petersson volume of the moduli space of hyperbolic surfaces
with boundaries was proved by Anantharaman-Monk \cite{AM}.

For $0\leq t\leq1$, put
\begin{align}
 H(s)&=\prod_{j=2}^{\infty}(1-s/j^2)^{-1},\qquad\qquad
 q_t(s)=\frac{H(s)^t}{1-s},\label{eq:family-q}\\
 y_t(z)&=\frac{z}{q_t(4z^2)}
 =z(1-4z^2)H(4z^2)^{-t}.\label{eq:family-y}
\end{align}
The powers are defined by the analytic branch equal to $1$ at the
origin. Let $V^{(t)}_{g,n}$ be the cusp coefficients of topological
recursion for $x=z^2/2$ and $y=y_t$, with the conventions of
Section~\ref{sec:exact}. Euler's product for the sine function gives
$y_0(z)=z-4z^3$ and $y_1(z)=\sin(2\pi z)/(2\pi)$.
The first positive zero is $a=1/2$ for every $t$. Define
\begin{equation}\label{eq:action}
 a=\frac12,\qquad
 A_t=\int_0^a2z y_t(z)\,\dd z,\qquad
 \kappa_t=-2a y_t'(a)=2^{1-t}.
\end{equation}

\begin{theorem}[Uniform variation]\label{thm:transport}
For each fixed $n\geq2$, set
\begin{equation}\label{eq:amplitude}
 \mathcal B_{g,n}(t)=
 \frac{V^{(t)}_{g,n}A_t^{2g+n-5/2}\sqrt{\kappa_t}}
 {\Gamma(2g+n-5/2)}.
\end{equation}
As $g\to\infty$,
\begin{equation}\label{eq:uniform-transport}
 \sup_{0\leq t\leq1}
 \left|\frac{\dd}{\dd t}\log\mathcal B_{g,n}(t)\right|
 =O_n(g^{-1}).
\end{equation}
Consequently, uniformly for $t,u\in[0,1]$,
\begin{equation}\label{eq:transport-ratio}
 \mathcal B_{g,n}(t)/\mathcal B_{g,n}(u)=1+O_n(g^{-1}).
\end{equation}
\end{theorem}

\begin{corollary}\label{cor:constant}
The constant in \eqref{eq:mz} is
\begin{equation}\label{eq:result}
 C_{\WP}=\frac1{\sqrt\pi}.
\end{equation}
\end{corollary}

\begin{theorem}[First correction]\label{thm:first-volume-correction}
For every fixed $n\geq2$, uniformly for $0\leq t\leq1$,
\begin{equation}\label{eq:main-refinement}
 \mathcal B_{g,n}(t)=\frac1{\sqrt2\pi^{3/2}}
 \left(1+\frac{\mathfrak b_n(t)}g+O_n(g^{-2})\right).
\end{equation}
Put $B_t=2^t$ and $M_t=3tB_t/4$. Then
\begin{align}
 \mathfrak b_2(t)={}&A_t\left[-\frac23+
 \left(1-\frac{\pi^2}{6}\right)t
 +2^t\left\{-\frac{29}{24}
 +\left(\frac{77}{32}-\frac{\pi^2}{8}\right)t
 -\frac{33}{32}t^2\right\}\right],\label{eq:main-b2}\\
 \mathfrak b_n(t)={}&\mathfrak b_2(t)
 +A_t(n-2)\left[\frac{B_t(6-n)}2+3-n-2M_t\right].\notag
\end{align}
\end{theorem}

Theorem~\ref{thm:transport} is proved in Section~\ref{sec:transport}.
We first obtain the descendant expansion in
Proposition~\ref{prop:first-correction}. Substitution in the dilaton
equation, followed by three moment identities, gives the terms
involving $\dot A_t$ and $\dot\kappa_t$. The remainder is summable
uniformly in $t$. Appendix~\ref{app:next} carries this calculation
one order further and proves Theorem~\ref{thm:first-volume-correction}.

At the cubic endpoint, the string equation and KdV reduce the
coefficients to those of Painlev\'e~I. The coefficient asymptotic
of Garoufalidis--Its--Kapaev--Mari\~no
\cite[Section 2.3, (2.39)]{GIKM} determines their leading amplitude.
Theorem~\ref{thm:transport} then gives the WP amplitude.
Corollary~\ref{cor:constant} follows from the normalization in
Section~\ref{sec:constant}, first
with two markings and then with every fixed number of markings.
The case $n=0$ requires a separate use of the dilaton equation.
The leading asymptotic holds for every fixed $n\geq0$; the
first-correction theorem is proved here for $n\geq2$.

We use the sine-curve identification with WP volumes and Eynard's
intersection-number formula \cite{EO,Eyn}, together with symmetry
of the stable recursion differentials. These identify the recursion
with a shifted Witten--Kontsevich specialization, to which the string,
dilaton, and KdV equations apply. The factorial upper bound used in
the estimates follows by induction from the exact cubic recurrence.
The Painlev\'e asymptotic theorem is used only to determine the
leading amplitude.

\

\noindent{\bf Acknowledgements} This work was supported by the
National Key Research and Development Program of China (NKPs) [Grant
Number 2024YFA1013201], the Natural Science Foundation of Chongqing
(NSFCQ) [Grant Number CSTB2024NSCQ-LZX0040], the Special Project of
Chongqing Municipal Science and Technology Bureau [Grant Number
2025CCZ015]. 

\

\noindent{\bf Disclosure of AI assistance.}
ChatGPT assisted with the proofs at an early 
stage, particularly in numerical experiments and proofreading.

\section{Recursion and shifted descendants}\label{sec:exact}
We use the bidifferential and recursion kernel
\begin{equation}\label{eq:kernel}
 B(z_1,z_2)=\frac{\dd z_1\dd z_2}{(z_1-z_2)^2},\qquad
 K_t(z_0,z)=-\frac{\dd z_0}
 {2y_t(z)(z_0^2-z^2)\dd z}.
\end{equation}
The stable differentials define volume polynomials by the Laplace transform
\begin{equation}\label{eq:laplace}
 \omega^{(t)}_{g,n}(\mathbf z)
 =\int_{\mathbb R_+^n}V^{(t)}_{g,n}(\mathbf L)
 \prod_{i=1}^n L_i e^{-L_i z_i}\,\dd L_i\,\dd z_i.
\end{equation}
When the parameter is understood, we write
$V_{g,n}=V^{(t)}_{g,n}(\bzero)$ and put
\begin{equation}\label{eq:J}
 J_{g,n}(\mathbf d)=
 \prod_{i=1}^n4^{d_i}(2d_i+1)!
 [\mathbf L^{2\mathbf d}]V^{(t)}_{g,n}(\mathbf L).
\end{equation}
Thus $J_{g,n}(\bzero)=V_{g,n}$. We assign the value zero to an
expression with a negative index, an inadmissible dimension, or an
unstable factor.

Write $q_t(s)=\sum_{L\geq0}b_L(t)s^L$ and let
$D=3g-3+n-|\mathbf d|$. Taking Laurent coefficients in the recursion
with kernel \eqref{eq:kernel} gives, except for $(0,3)$ and $(1,1)$,
\begin{align}
 J_{g,n}(\mathbf d)
 ={}&4\sum_{j=2}^n(2d_j+1)\sum_{L=0}^{D}b_L
 J_{g,n-1}(L+d_1+d_j-1,\mathbf d_{\widehat1,\widehat j})
 \label{eq:rec}\\
 &+8\sum_{L=0}^{D}b_L\sum_{u+v=L+d_1-2}
 \left[J_{g-1,n+1}(u,v,\mathbf d_{\widehat1})\right.\notag\\
 &\hspace{38mm}\left.
 +\sum_{\substack{g_1+g_2=g\\I\sqcup J=\{2,\ldots,n\}}}^{\st}
 J_{g_1,|I|+1}(u,\mathbf d_I)
 J_{g_2,|J|+1}(v,\mathbf d_J)\right].\notag
\end{align}
The sum over stable splittings is ordered. The initial values are
\begin{equation}\label{eq:initial}
 J_{0,3}(0,0,0)=1,\qquad
 J_{1,1}(0)=\frac{b_1}2,\qquad J_{1,1}(1)=\frac12.
\end{equation}
At the WP endpoint, $b_L=2a_L$ in the notation of Mirzakhani--Zograf;
thus $4b_L$ and $8b_L$ become $8a_L$ and $16a_L$.

We verify \eqref{eq:rec} directly. The term of the stable
differential indexed by $\mathbf d$ is
\[
 4^{-|\mathbf d|}J_{g,n}(\mathbf d)
 \prod_i\frac{\dd z_i}{z_i^{2d_i+2}}.
\]
The kernel has the expansion
\[
 K_t=-\frac12\sum_{L,m\geq0}b_L4^Lz^{2L+2m-1}
 \frac{\dd z_0}{z_0^{2m+2}\dd z}.
\]
In a term with two node markings, the residue condition is
$u+v=L+d_1-2$. The sign from $\dd(-z)$ cancels the sign of the
kernel, and the normalization factor is $4^{L+d_1-u-v}/2=8$.
In a merging term, the relevant Bergman coefficient is
$(2d_j+1)z^{2d_j}/z_j^{2d_j+2}$. The two placements of the unstable
Bergman factor cancel the factor $1/2$. The residue condition is
$u=L+d_1+d_j-1$, and the resulting factor is $4(2d_j+1)$.
The dimension excess of a merging or nonseparating term is $D-L$;
for a stable splitting, the two excesses have the same sum.
Hence every term has the common range $0\leq L\leq D$.

Define the shifted descendant times $s_k(t)$ by
\begin{equation}\label{eq:times}
 y_t(z)=z-\sum_{k\geq2}s_k(t)\frac{z^{2k-1}}{(2k-1)!!},
 \qquad p_k=4^k(2k+1)!!.
\end{equation}
The one-branch-point formula gives
\begin{equation}\label{eq:shifted}
 F_{g,n}(\mathbf d):=
 \left\langle\prod_{i=1}^n\tau_{d_i}
 \exp\left(\sum_{k\geq2}s_k(t)\tau_k\right)\right\rangle_g,
 \qquad
 J_{g,n}(\mathbf d)=\prod_i p_{d_i}\,F_{g,n}(\mathbf d),
\end{equation}
where $p_0=1$. For fixed $g,n,\mathbf d$, this is a finite polynomial
in the times: each additional $\tau_k$ uses $k-1\geq1$ units of
dimension.

To fix the normalization, consider the formal Gaussian transform
\begin{equation}\label{eq:gaussian}
 \sqrt{\frac2\pi}\,u^{3/2}
 \int_0^\infty z y_t(z)e^{-uz^2/2}\,\dd z
 =1-\sum_{k\geq2}s_k(t)u^{-(k-1)}.
\end{equation}
The integral is interpreted by termwise Gaussian integration of the
series at the branch point. The corresponding kappa-time series is
$-\log(1-\sum_{k\geq2}s_k v^{k-1})$. The forgetful-map identity
expresses these kappa insertions as the additional descendant markings
in \eqref{eq:shifted}. The leading Gaussian coefficient is one, and
$B$ has no regular part in the coordinate $z$. There is therefore no
Euler-characteristic factor or change of basis at the external markings.

We use \cite[Theorem 3.2, equations (3.10)--(3.12)]{Eyn}, whose
conventions are $x=z^2$ and
$K_E=\dd z_0/[4y(z)(z_0^2-z^2)\dd z]$.
Since our kernel is $-2K_E$, homogeneity of the recursion gives
$\omega_{g,n}=(-2)^{2g-2+n}W^{(g)}_n$ in that notation. This cancels
the factor $(-2)^{2-2g-n}$ in the cited formula.
The coefficient denoted there by $t_{2j+3}$ is
$[z^{2j+1}]y(z)$, so the Schur transform is
$1-\sum_{k\geq2}s_kv^{k-1}$. These identifications give the signs and
powers of two in \eqref{eq:shifted}. A fixed correlator involves only
finitely many coefficients of $y$; hence the theorem for polynomial
curves applies to our analytic family by Taylor truncation.

For example, the first nonlinear time gives
\begin{equation}\label{eq:smallcheck}
 s_2=12b_1,\quad s_3=240(b_2-b_1^2),\quad
 V_{0,4}=12b_1,\quad V_{1,2}=2b_1^2+10b_2.
\end{equation}
All quantities are evaluated at the same $t$, which we suppress.
Indeed, \eqref{eq:shifted} gives $V_{1,2}=s_2^2/12+s_3/24$, since
$\langle\tau_0^2\tau_2^2\rangle_1=1/6$ and
$\langle\tau_0^2\tau_3\rangle_1=1/24$.

\section{Positivity and exact identities}
Unless otherwise stated, constants below may depend on the fixed
numbers of markings, but not on $t\in[0,1]$ or on sufficiently large
$g$.

\begin{lemma}[Bounds for the kernel coefficients]\label{lem:kernel}
Set $b_{-1}=0$, $\Delta_L=b_L-b_{L-1}$, $B_t=2^t$, and
$M_t=3tB_t/4$. Then
\begin{gather}
 1=b_0\leq b_1\leq\cdots\leq B_t\leq2,
 \qquad \Delta_L\geq0,\label{eq:positive}\\
 \sum_{L\geq0}\Delta_L=B_t,\qquad
 \sum_{L\geq0}L\Delta_L=M_t,\qquad
 \sup_t\sum_{L\geq0}(1+L)^m\Delta_L<\infty
 \quad(m\geq0).\label{eq:moments-delta}
\end{gather}
The tails decay exponentially, uniformly in $t$.
\end{lemma}
\begin{proof}
Write $G_t(s)=H(s)^t=\sum_{L\geq0}a_L(t)s^L$. Since
$q_t(s)=G_t(s)/(1-s)$, multiplication by
$(1-s)^{-1}=\sum_{r\geq0}s^r$ gives, coefficient by coefficient,
\[
 b_L(t)=\sum_{r=0}^{L}a_r(t),\qquad
 \Delta_L(t)=b_L(t)-b_{L-1}(t)=a_L(t)=[s^L]G_t(s).
\]
This also holds for $L=0$, since $a_0(t)=G_t(0)=1$ and $b_{-1}=0$.

For $|s|<4$, expand each logarithm in the product for $H$:
\begin{align*}
 \log H(s)
 &=-\sum_{j\geq2}\log\left(1-\frac{s}{j^2}\right)
   =\sum_{j\geq2}\sum_{m\geq1}\frac{s^m}{m j^{2m}}\\
 &=\sum_{m\geq1}c_m s^m,\qquad
 c_m:=\frac{\zeta(2m)-1}{m}
      =\frac1m\sum_{j\geq2}\frac1{j^{2m}}>0.
\end{align*}
The series is absolutely and locally uniformly convergent in $|s|<4$.
It follows that
\[
 G_t(s)=\exp\left(t\sum_{m\geq1}c_m s^m\right).
\]
More explicitly, for $L\geq1$,
\[
 a_L(t)=
 \sum_{\substack{k_1,\ldots,k_L\geq0\\
                  k_1+2k_2+\cdots+Lk_L=L}}
 \prod_{m=1}^{L}\frac{(t c_m)^{k_m}}{k_m!},
 \qquad a_0(t)=1.
\]
For fixed $L$ this is a finite sum of monomials in $t$ with nonnegative
coefficients. Hence $a_L(t)=\Delta_L(t)\geq0$ and $\Delta_L(t)$ is
nondecreasing in $t$. Therefore $b_L(t)=\sum_{r=0}^{L}\Delta_r(t)$ is
nondecreasing in $L$, and $b_0(t)=1$.

The product at $s=1$ can be evaluated by finite telescoping:
\[
 \prod_{j=2}^{N}\left(1-\frac1{j^2}\right)^{-1}
 =\prod_{j=2}^{N}\frac{j}{j-1}\prod_{j=2}^{N}\frac{j}{j+1}
 =N\frac{2}{N+1}=\frac{2N}{N+1}.
\]
Thus $H(1)=2$. Logarithmic differentiation, justified by local uniform
convergence of the product and its logarithmic derivative near $s=1$,
gives
\[
 \frac{H'(1)}{H(1)}
 =\sum_{j\geq2}\frac1{j^2-1}.
\]
The partial fractions telescope:
\[
 \sum_{j=2}^{N}\frac1{j^2-1}
 =\frac12\sum_{j=2}^{N}
       \left(\frac1{j-1}-\frac1{j+1}\right)
 =\frac12\left(\frac32-\frac1N-\frac1{N+1}\right)
 \longrightarrow\frac34.
\]
Consequently $G_t(1)=2^t=B_t$ and
$G_t'(1)=tG_t(1)H'(1)/H(1)=3tB_t/4=M_t$. Because $G_t$ is analytic
in a disk of radius $4$, its Taylor series and differentiated series
converge absolutely at $s=1$. Termwise evaluation therefore yields
\[
 \sum_{L\geq0}\Delta_L(t)=G_t(1)=B_t,\qquad
 \sum_{L\geq0}L\Delta_L(t)=G_t'(1)=M_t.
\]
The first identity and nonnegativity give $b_L(t)\leq B_t$; also
$1\leq B_t=2^t\leq2$. This proves the inequalities and the two moment
identities in the lemma.

It remains to prove the uniform moment and tail bounds. Fix
$1<\rho<4$. The power series for $\log H(s)$ converges uniformly on
$|s|\leq\rho$, so
\[
 C_\rho:=\exp\left(\max_{|s|=\rho}|\log H(s)|\right)<\infty
 \quad\hbox{and}\quad
 \sup_{0\leq t\leq1}\max_{|s|=\rho}|G_t(s)|\leq C_\rho.
\]
Cauchy's coefficient estimate then gives
\[
 0\leq\Delta_L(t)\leq C_\rho\rho^{-L}
 \qquad(L\geq0,\ 0\leq t\leq1).
\]
For every fixed $m\geq0$, comparison with the convergent geometric series
implies
\[
 \sup_{0\leq t\leq1}\sum_{L\geq0}(1+L)^m\Delta_L(t)
 \leq C_\rho\sum_{L\geq0}(1+L)^m\rho^{-L}<\infty.
\]
More quantitatively, choose $1<\rho_0<\rho<4$ and set
\[
 K_{m,\rho_0,\rho}
 :=\sup_{L\geq0}(1+L)^m\left(\frac{\rho_0}{\rho}\right)^L<\infty.
\]
Then, for every $N\geq0$,
\[
 \sup_{0\leq t\leq1}\sum_{L\geq N}(1+L)^m\Delta_L(t)
 \leq\frac{C_\rho K_{m,\rho_0,\rho}}{1-\rho_0^{-1}}\rho_0^{-N}.
\]
This is the asserted exponential decay, uniformly in $t$.
\end{proof}

For stable $(g,n)$ with $n\geq1$, induction in \eqref{eq:rec} shows
that the coefficients are polynomials in the $b_L$ with nonnegative
coefficients. Consequently,
\begin{equation}\label{eq:endpoint-upper}
 0<V^{(0)}_{g,n}\leq V^{(t)}_{g,n}\leq V^{\WP}_{g,n},
 \qquad n\geq1,\quad 2g-2+n>0.
\end{equation}
This positivity does not require the shifted times $s_k(t)$ to have
the same sign. The restriction $n\geq1$ is necessary here, since the
recursion requires a distinguished marking. Equation
\eqref{eq:endpoint-upper} makes no assertion about $V^{(t)}_{g,0}$.
Its endpoint identification and positivity for large genus will follow
from Proposition~\ref{prop:all-n}.

\begin{lemma}[Comparison with the cubic curve]\label{lem:cubic-comparison}
For every stable $(g,n)$ with $n\geq1$ and all admissible degrees,
\begin{equation}\label{eq:cubic-comparison}
 J^{(0)}_{g,n}(\mathbf d)\leq J^{(t)}_{g,n}(\mathbf d)
 \leq B_t^{2g-2+n}J^{(0)}_{g,n}(\mathbf d).
\end{equation}
For each fixed $k\geq1$, there are constants $C,P$ such that
\begin{equation}\label{eq:cubic-upper}
 V^{(t)}_{h,k}\leq
 C\left(\frac{B_t}{A_0}\right)^{2h}4^h(h!)^2h^P,
 \qquad h\geq1,\quad 0\leq t\leq1.
\end{equation}
\end{lemma}
\begin{proof}
Regard $b_0$ temporarily as a formal variable in the recursion, with
initial values $J_{0,3}(\bzero)=b_0$, $J_{1,1}(1)=b_0/2$, and
$J_{1,1}(0)=b_1/2$. Each coefficient is a homogeneous polynomial in
the $b_L$, with nonnegative coefficients and ordinary degree
$2g-2+n$. Indeed, the degree in a merging or nonseparating term is
one less, and the two degrees in a product have the same sum.
Multiplication by the kernel coefficient raises the degree by one,
which proves the assertion by induction.

For the cubic kernel, $b_L=1$ for every $L$, whereas
$1\leq b_L(t)\leq B_t$. Evaluating these homogeneous polynomials
gives \eqref{eq:cubic-comparison}. 

The string and KdV equations give the following identities for the
cubic curve, as proved in Section~\ref{sec:anchor}:
\begin{align}
 V^{(0)}_{h,1}&=\frac{c_h}{12(5h-3)},\label{eq:cubic-fixedn}\\
 V^{(0)}_{h,k}&=c_h12^{k-2}
 \prod_{j=0}^{k-3}(5h-1+2j),\qquad k\geq2,\notag
\end{align}
where the empty product for $k=2$ equals one. The recurrence
\eqref{eq:c-rec} implies
\begin{equation}\label{eq:cubic-elementary}
 c_h\leq\frac1{300}\,3600^h(h!)^2
 =\frac1{300}\,A_0^{-2h}4^h(h!)^2.
\end{equation}
To prove this, put $a_h=c_h/[3600^h(h!)^2]$. Then $a_1=1/300$,
and the recurrence becomes
\[
 a_h=\left(1-\frac2h+\frac{24}{25h^2}\right)a_{h-1}
 +6\sum_{j=1}^{h-1}\frac{a_ja_{h-j}}{\binom hj^2}.
\]
Since $\binom hj\geq h$ for $1\leq j\leq h-1$, the induction
hypothesis $a_j\leq1/300$ bounds the multiplier of $1/300$ by
\[
 1-\frac2h+\frac{24}{25h^2}
 +\frac1{50}\frac{h-1}{h^2}
 =1+\frac{-99h+47}{50h^2}\leq1.
\]
This proves \eqref{eq:cubic-elementary}. The factors multiplying $c_h$
in \eqref{eq:cubic-fixedn} are polynomial in $h$, and $B_t^{k-2}$
is uniformly bounded. Thus \eqref{eq:cubic-upper} follows. This
argument uses only the exact cubic recurrence; it uses neither the
transport theorem nor the Painlev\'e asymptotic formula.
\end{proof}

\begin{lemma}[Monotonicity and adjacent differences]\label{lem:mono}
For every stable $(g,n)$ with $n\geq1$ and all nonnegative degrees,
\begin{equation}\label{eq:monotone}
 0\leq J_{g,n}(\mathbf d+\mathbf e_i)
 \leq J_{g,n}(\mathbf d)\leq V_{g,n}.
\end{equation}
Except in the two initial cases, the adjacent difference is the
right-hand side of \eqref{eq:rec} with $b_L$ replaced by $\Delta_L$,
and with the original degrees and upper limit $D$.
\end{lemma}
\begin{proof}
By symmetry, it is enough to take $i=1$. Put $r=L+d_1$.
Increasing $d_1$ by one decreases $D$ by one, so the upper limit
$r\leq D+d_1$ is unchanged. At fixed $r$, all indices on the
right-hand side remain unchanged, while the coefficient changes
from $b_{r-d_1}$ to $b_{r-d_1-1}$. The difference is
$\Delta_{r-d_1}$; the convention $b_{-1}=0$ includes the lowest term.
Positivity proves the adjacent inequality. The initial values
\eqref{eq:initial} satisfy it as well, since $b_1\geq1$. Iteration
gives the bound by $V_{g,n}$.
\end{proof}

Applying the dilaton equation to \eqref{eq:shifted} gives
\begin{equation}\label{eq:dilaton}
 (2g-2+n)V_{g,n}
 =\sum_{k\geq1}w_k(t)J_{g,n+1}(k,0^n),\qquad
 w_1=\frac1{12},\quad w_k=-\frac{s_k}{p_k}\ (k\geq2).
\end{equation}
The number of additional markings in the dilaton equation contributes
$\sum s_k\partial_{s_k}$; moving this term to the right gives the
stated weights. For any fixed $1/2<R<1$, Cauchy's estimates applied
to \eqref{eq:times} give
\begin{equation}\label{eq:weights}
 \sup_t\sum_{k\geq1}(1+k)^m|w_k(t)|<\infty,
 \qquad
 \sup_t\sum_{k\geq2}(1+k)^m\frac{|\dot s_k(t)|}{p_k}<\infty.
\end{equation}
Indeed, the coefficient bounds contain the geometric factor
$(4R^2)^{-k}/(2k+1)$. Both $y_t$ and $\dot y_t$ are uniformly
analytic on this circle.

We shall also use the following integrated form of the first KdV
equation:
\begin{align}
 \left\langle\tau_1\tau_0\prod_{i=1}^m\tau_{d_i}\right\rangle_g
 ={}&\frac1{12}\left\langle\tau_0^4\prod_{i=1}^m\tau_{d_i}
 \right\rangle_{g-1}\label{eq:kdv-coefficients}\\
 &+\frac12\sum_{\substack{h+h'=g\\I\sqcup J=\{1,\ldots,m\}}}
 \left\langle\tau_0^2\prod_{i\in I}\tau_{d_i}\right\rangle_h
 \left\langle\tau_0^2\prod_{i\in J}\tau_{d_i}\right\rangle_{h'}.
 \notag
\end{align}
Here and below an unstable intersection number is zero. This identity
follows from the formula in \cite[Section 5, proof of Proposition 5.1]{LX}
by the dilaton equation. In particular, it fixes the constant of
integration in KdV.

To specialize \eqref{eq:kdv-coefficients}, take $n-2$ of the insertions
to be $\tau_0$, and sum the remaining insertions with weights
$\prod_k s_k^{m_k}/m_k!$. In the product term, if $a_k$ of the $m_k$
insertions of degree $k$ lie on the first factor, the number of choices
is $\binom{m_k}{a_k}$. Thus
\[
 \frac{s_k^{m_k}}{m_k!}\binom{m_k}{a_k}
 =\frac{s_k^{a_k}}{a_k!}\frac{s_k^{m_k-a_k}}{(m_k-a_k)!},
\]
and the two factors acquire their own shifted exponentials.
Choosing $r$ of the distinguished $n-2$ zero insertions for the first
factor gives $\binom{n-2}{r}$. Finally $p_1=12$ and $p_0=1$, so
\begin{equation}\label{eq:kdv}
 J_{g,n}(1,0^{n-1})=V_{g-1,n+2}
 +6\sum_{h=0}^g\sum_{r=0}^{n-2}\binom{n-2}{r}
 V_{h,r+2}V_{g-h,n-r},\qquad n\geq2.
\end{equation}
The splitting is ordered, as in \eqref{eq:kdv-coefficients}.
For a genus-zero shifted coefficient with $r=1$ or $2$ zero
insertions, a nonzero term would require
\[
 \sum_{j=1}^m(k_j-1)=r-3<0,\qquad k_j\geq2.
\]
Hence $V_{0,1}=V_{0,2}=0$. Stable genus-zero factors remain in
\eqref{eq:kdv}; in particular, $V_{0,3}=1$.

Taking absolute values in \eqref{eq:dilaton} and applying
\eqref{eq:monotone} and \eqref{eq:weights}, we obtain a constant
$C_*$ such that
\begin{equation}\label{eq:lower-marking}
 (2g-2+n)V_{g,n}\leq C_*V_{g,n+1}.
\end{equation}
This is a lower bound for the increase in volume upon adding a marking.
Here $n\geq1$; no positivity or growth bound for $V_{g,0}$ has been
used. Positivity in \eqref{eq:kdv} also gives
\begin{equation}\label{eq:genus-inequality}
 V_{g-1,n+2}\leq V_{g,n},\qquad n\geq2.
\end{equation}

\section{Coarse estimates}
We first establish the estimates needed below without using an
asymptotic formula for the genus ratio.

\begin{lemma}[Coarse separating estimate]\label{lem:coarse-split}
For fixed $n\geq2$ and $k,b\geq1$ with $k+b=n+1$,
\begin{equation}\label{eq:coarse-split}
 \sum_h^{\st}V_{h,k}V_{g-h,b}=O_n(V_{g,n}/g).
\end{equation}
The estimate is uniform for $t\in[0,1]$.
\end{lemma}
\begin{proof}
Suppose first that $k,b\geq2$. Then $n\geq3$, and the product occurs
in \eqref{eq:kdv} with $n-1$ markings and a fixed positive binomial
coefficient. Its sum is therefore $O_n(V_{g,n-1})$, which is
$O_n(V_{g,n}/g)$ by \eqref{eq:lower-marking}.

By symmetry, the remaining case is
$\sum_{h\geq1}V_{h,1}V_{g-h,n}$. Choose a sufficiently small
constant $\epsilon>0$. For $h\geq\epsilon g$, the lower marking
estimate gives
\[
 V_{h,1}\leq\frac{C_*}{2h-1}V_{h,2}
 \leq\frac{C_\epsilon}{g}V_{h,2}.
\]
The product terms in \eqref{eq:kdv} with $n$ markings then bound
this part of the sum by $O_n(V_{g,n}/g)$. The same argument includes
the term in which the other component has genus zero, when it is stable.

For $1\leq h<\epsilon g$, two applications of the lower marking
estimate and \eqref{eq:genus-inequality} give
\[
 V_{j-1,n}\leq
 \frac{C_*^2}{(2j+n-4)(2j+n-3)}V_{j-1,n+2}
 \leq\frac{C_n}{j^2}V_{j,n}.
\]
Iteration yields
$V_{g-h,n}/V_{g,n}\leq(C'/g^2)^h$ in this range.
By the cubic upper bound \eqref{eq:cubic-upper}, with the constant
enlarged to include finitely many genera,
$V_{h,1}\leq C_0(C_1h)^{2h}$ for $h\geq1$. Thus this part of the sum,
divided by $V_{g,n}$, is at most
\[
 C_0\sum_{1\leq h<\epsilon g}(Kh/g)^{2h}.
\]
Choose $\epsilon$ so that $K\epsilon\leq1/2$. Each summand is at most
$K^2h^2g^{-2}4^{-(h-1)}$, and the sum is $O(g^{-2})$.
All constants may be chosen independently of $t$.
\end{proof}

\begin{proposition}[Descendant deficit]\label{prop:deficit}
For fixed $n\geq2$, all $g\geq1$, all nonnegative degrees, and all
$t\in[0,1]$,
\begin{equation}\label{eq:deficit}
 0\leq1-\frac{J_{g,n}(\mathbf d)}{V_{g,n}}
 \leq\min\left\{1,\frac{C_n|\mathbf d|^2}{g}\right\}.
\end{equation}
Here $J_{g,n}(\mathbf d)$ is zero beyond the dimension, as stipulated
above. At $\mathbf d=\bzero$, the deficit is zero.
\end{proposition}
\begin{proof}
The two cusp ratios occurring in the defect recursion satisfy
\begin{equation}\label{eq:coarse-child}
 V_{g,n-1}=O_n(V_{g,n}/g),\qquad
 V_{g-1,n+1}=O_n(V_{g,n}/g).
\end{equation}
The first follows from the lower marking estimate, including the case
of a parent with one marking. For the second, add a marking at genus
$g-1$ and apply \eqref{eq:genus-inequality}.

In the defect recursion, bound each descendant on the right-hand side
by its cusp volume. The merge multiplicities have sum
$O_n(1+|\mathbf d|)$, and the two-node index sum has
$O(1+d_1+L)$ pairs. Equations \eqref{eq:moments-delta} and
\eqref{eq:coarse-child}, together with
Lemma~\ref{lem:coarse-split}, give
\[
 0\leq J_{g,n}(\mathbf d)-J_{g,n}(\mathbf d+\mathbf e_i)
 \leq\frac{C_n(1+|\mathbf d|)}g V_{g,n}.
\]
Sum this inequality along a path from $\bzero$ to $\mathbf d$ to
obtain the quadratic bound. Monotonicity gives the bound by one.
Beyond the dimension, the quadratic bound is at least one for large
$g$; enlarging $C_n$ covers the remaining cases. The argument uses only
upper bounds by cusp volumes, not asymptotic estimates for the
descendants on the right-hand side.
\end{proof}

\section{Marking ratios and separating estimates}
The first positive root of $y_t$ is $a=1/2$ for every $t\in[0,1]$.
Euler's sine product gives $y_1(z)=\sin(2\pi z)/(2\pi)$, and hence
\begin{equation}\label{eq:actions}
 A_1=A_{\WP}=\frac1{4\pi^2},\qquad A_0=\frac1{30},\qquad
 A_{\WP}\leq A_t\leq A_0.
\end{equation}
The inequalities follow from $H(4z^2)\geq1$ for $0\leq z\leq a$.
Termwise integration in \eqref{eq:times} gives
\begin{equation}\label{eq:sumw}
 \sum_{k\geq1}w_k=A_t.
\end{equation}

\begin{lemma}[Marking ratio]\label{lem:marking}
For every fixed parent marking count $n\geq1$,
\begin{equation}\label{eq:marking}
 \frac{V_{g,n+1}}{V_{g,n}}
 =\frac{2g-2+n}{A_t}+O_n(1)
\end{equation}
uniformly for $t\in[0,1]$.
\end{lemma}
\begin{proof}
Apply \eqref{eq:deficit} with $n+1\geq2$ markings in
\eqref{eq:dilaton}. Since $\sum_{k\geq1}k^2|w_k|$ is uniformly bounded,
\[
 (2g-2+n)\frac{V_{g,n}}{V_{g,n+1}}=A_t+O_n(g^{-1}).
\]
The lower bound $A_t\geq A_{\WP}>0$ permits uniform inversion and
proves \eqref{eq:marking}.
\end{proof}

\begin{lemma}[Factorial convolution]\label{lem:binomial}
For fixed $P\geq0$ and $0\leq r_*<2$,
\begin{equation}\label{eq:binomial}
 \sum_{h=1}^{\lfloor g/2\rfloor}
 \frac{h^P r_*^{2h}}{\binom gh^2}=O_{P,r_*}(g^{-2}).
\end{equation}
\end{lemma}
\begin{proof}
Choose an integer $H>P+3$. The fixed-$h$ factorial expressions show
that the terms with $h\leq H$ have sum $O(g^{-2})$.
For $H<h\leq\sqrt g$, the inequality
$\binom gh\geq(g/h)^h$ gives
\[
 \frac{h^P r_*^{2h}}{\binom gh^2}
 \leq h^P(r_*h/g)^{2h}
 \leq h^P(r_*^2/g)^h.
\]
For $q=r_*^2/g\leq1/2$,
$\sum_{h>H}h^Pq^h\leq C_{H,P}q^{H+1}$, so this part of the sum is
$O(g^{-H-1})$. For $\sqrt g<h\leq g/2$, use
$\binom gh\geq2^h$. The remaining sum is bounded by
$\sum_{h>\sqrt g}h^P(r_*/2)^{2h}$, which is exponentially small in
$\sqrt g$.
\end{proof}

\begin{proposition}[Refined separating estimates]\label{prop:refined}
Uniformly for $t\in[0,1]$, and for fixed $n\geq2$,
\begin{align}
 \sum_h^{\st}V_{h,k}V_{g-h,b}
 &=O_n(V_{g,n}/g^2),& k+b&=n+1,\quad k,b\geq1,
 \label{eq:split-fine}\\
 \sum_h^{\st}V_{h,k}V_{g-h,b}
 &=O_n(V_{g,n}/g),& k+b&=n+2,\quad k,b\geq2.
 \label{eq:kdv-split-fine}
\end{align}
\end{proposition}
\begin{proof}
Fix a reference marking count $N\geq2$. Two applications of
Lemma~\ref{lem:marking} and the one-sided inequality
\eqref{eq:genus-inequality} give
\[
 \frac{V_{j-1,N}}{V_{j,N}}
 \leq\frac{A_t^2}{4j^2}\bigl(1+O_N(j^{-1})\bigr).
\]
For $1\leq h\leq g/2$, iteration produces a product of relative
errors bounded by $C(g/(g-h))^C$, and therefore by a constant.
It follows that
\begin{equation}\label{eq:genus-upper-iterate}
 \frac{V_{g-h,N}}{V_{g,N}}
 \leq C\frac{A_t^{2h}}{4^h}
 \left(\frac{(g-h)!}{g!}\right)^2.
\end{equation}
Only the one-sided genus inequality is used here.

For each fixed $k$, the homogeneous comparison with the cubic model
gives
\[
 V_{h,k}\leq
 C(B_t/A_0)^{2h}4^h(h!)^2h^P,\qquad h\geq1.
\]
Set $r(t)=B_tA_t/A_0$. Then
\begin{equation}\label{eq:new-action-gap}
 r(t)\leq r(1)=\frac{2A_1}{A_0}=\frac{15}{\pi^2}<2.
\end{equation}
Indeed,
\[
 B_tA_t=\int_0^{1/2}2z^2(1-4z^2)
 \left(\frac2{H(4z^2)}\right)^t\dd z
\]
is increasing in $t$, since $1\leq H(s)\leq H(1)=2$ for
$0\leq s\leq1$. Applying Lemma~\ref{lem:marking} to the fixed changes
in marking count of the component of larger genus, we obtain
\begin{equation}\label{eq:product-fine}
 \frac{V_{h,k}V_{g-h,b}}{V_{g,n}}
 \leq C h^P g^{b-n}\frac{r_*^{2h}}{\binom gh^2},
 \quad 1\leq h\leq g/2,\qquad
 r_*:=\frac{15}{\pi^2}<2.
\end{equation}
When $b=1$, first compare with $b=2$ by the lower marking estimate,
which contributes a factor $O(g^{-1})$ and gives the same formula.
For the smaller component, the cubic bound also applies when $k=1$.

Apply Lemma~\ref{lem:binomial}, and interchange the two components
to treat the other half of the genus sum. If $k+b=n+1$, then
$b-n=1-k\leq0$, and each half with both genera positive is
$O(g^{-2})$ after division by $V_{g,n}$. A stable component of genus
zero has $k\geq3$; the marking ratio bounds its contribution by
$O(g^{1-k})=O(g^{-2})$. This proves \eqref{eq:split-fine}.

If $k+b=n+2$ and $k,b\geq2$, the parts with both genera positive
are again $O(g^{-2})$. A stable component of genus zero with $k=3$
contributes $O(g^{2-k})=O(g^{-1})$, and those with $k>3$ contribute
less. This proves \eqref{eq:kdv-split-fine}.
\end{proof}

Equations \eqref{eq:kdv}, \eqref{eq:deficit}, and
\eqref{eq:kdv-split-fine} now imply
\begin{equation}\label{eq:genus-leading}
 \frac{V_{g-1,n+2}}{V_{g,n}}=1+O_n(g^{-1}),\qquad n\geq2.
\end{equation}
Thus the asymptotic genus ratio follows from the separating estimates;
their proof required only the one-sided inequality
\eqref{eq:genus-inequality}.

\section{The first descendant correction}
\begin{proposition}[First correction with a summable remainder]
\label{prop:first-correction}
For every fixed $n\geq3$, uniformly for $t\in[0,1]$ and integers
$k\geq1$,
\begin{equation}\label{eq:E-expansion}
 \frac{J_{g,n}(k,0^{n-1})}{V_{g,n}}
 =1-\frac{E_{n,k}(t)}g
 +O_n\left(\frac{(k+1)^4}{g^2}\right),
\end{equation}
where
\begin{equation}\label{eq:E}
 E_{n,k}=2A_t\left[B_t k^2+\bigl(B_t(n-4)+2M_t\bigr)k-(n-3)\right].
\end{equation}
At $k=0$, the ratio is exactly one; we set $E_{n,0}=0$.
\end{proposition}
\begin{proof}
The marking ratio and \eqref{eq:genus-leading} give
\begin{equation}\label{eq:child-leading}
 \frac{V_{g,n-1}}{V_{g,n}}=\frac{A_t}{2g}+O_n(g^{-2}),\qquad
 \frac{V_{g-1,n+1}}{V_{g,n}}=\frac{A_t}{2g}+O_n(g^{-2}).
\end{equation}
Consider the difference between first degrees $d$ and $d+1$, with all
other degrees zero. In the exact defect of Lemma~\ref{lem:mono},
replace the descendants in the merge and nonseparating terms by their
cusp values. Their total degrees are $O(d+L+1)$. By
Proposition~\ref{prop:deficit} and \eqref{eq:child-leading}, the error
is at most
\[
 \frac{C_n}{g^2}\sum_{L\geq0}\Delta_L(d+L+1)^3
 =O_n((d+1)^3/g^2).
\]
The power three includes the number of two-node index pairs. Both
descendants have at least two markings, since $n\geq3$.
The refined separating estimate contributes
$O_n((d+1)/g^2)$ after summing over the index pairs and the kernel
moments. For $d\leq\sqrt g$, the finite $L$ sum may be extended to
infinity with an exponentially small error, since
$D=3g-3+n-d$.

When $d=0$, a merge with a zero-degree partner requires $L\geq1$,
so the sum of its kernel coefficients is $B_t-1$. For the
nonseparating term, the sum weighted by the number of index pairs is
\[
 \sum_{L\geq0}(L-1)_+\Delta_L=M_t-B_t+1.
\]
Hence the leading term of the adjacent defect, multiplied by
$g/V_{g,n}$, is
\begin{equation}\label{eq:d0}
 2A_t(n-1)(B_t-1)+4A_t(M_t-B_t+1).
\end{equation}
When $d\geq1$, the corresponding sums are $B_t$ and
$M_t+(d-1)B_t$. The leading term is then
\begin{equation}\label{eq:dpositive}
 2A_t(n-1)B_t+4A_t\bigl(M_t+(d-1)B_t\bigr).
\end{equation}
Summing \eqref{eq:d0} and \eqref{eq:dpositive} over $0\leq d<k$
gives \eqref{eq:E}. The sum of the errors is
$O_n((k+1)^4/g^2)$ for $k\leq\sqrt g$.

For $k>\sqrt g$, monotonicity gives $0\leq J/V\leq1$, while
$E_{n,k}=O_n((k+1)^2)$ uniformly in $t$. Thus the absolute error in
\eqref{eq:E-expansion} is at most $C_n(1+k^2/g)$, which is bounded
by $C_n'(k+1)^4/g^2$ in this range. The same argument applies beyond
the dimension, where the descendant is zero.
\end{proof}

At $t=1$, the correction is
\[
 E_{n,k}(1)=\frac{k^2+(n-5/2)k-(n-3)/2}{\pi^2}.
\]
At $t=0$, it is
$E_{n,k}(0)=[k^2+(n-4)k-(n-3)]/15$.
In particular, $E_{n,1}(0)=0$, in agreement with the exact cubic
identity $J_{g,n}(1,0^{n-1})=V_{g,n}$.

\section{Variation of the normalized volumes}\label{sec:transport}
We first record the moment identities needed in the proof of
Theorem~\ref{thm:transport}. Put
$\beta_t=a y_t''(a)/y_t'(a)$. Equation~\eqref{eq:family-y} and
Lemma~\ref{lem:kernel} give
\begin{equation}\label{eq:beta}
 B_t\kappa_t=2,\qquad
 \beta_t=3(1-t)=3-4M_t/B_t.
\end{equation}
The dilaton weights satisfy
\begin{equation}\label{eq:S}
 S_0:=\sum_{k\geq1}w_k=A_t,\qquad
 S_1:=\sum_{k\geq1}kw_k=-A_t/2,\qquad
 S_2:=\sum_{k\geq1}k^2w_k=A_t/4-\kappa_t/16.
\end{equation}
To prove these identities, regard the upper endpoint of the action
integral as a variable and write
\[
 A(a)=\frac{2a^3}{3}
 -2\sum_{k\geq2}\frac{s_k a^{2k+1}}{(2k+1)!!}.
\]
At $a=1/2$, the terms of this series are $w_k$. If
$D_a=a\partial_a$, then $(D_a-1)/2$ multiplies the $k$th term by
$k$. At the zero of $y$,
$D_aA=2a^2y(a)=0$ and
$D_a^2A=2a^3y'(a)=-\kappa_t/4$.
Applying these identities once and twice proves \eqref{eq:S}.

\subsection{Proof by the moment identities}
\begin{proof}[Proof of Theorem~\ref{thm:transport}]
Fix $n\geq2$ and suppress $t$ from the notation. Put
\[
 A=A_t,\qquad \kappa=\kappa_t,\qquad
 B=B_t=2^t,\qquad M=M_t=\frac34tB,\qquad
 N=2g+n-2.
\]
Set
\[
 r_k=\frac{J_{g,n+1}(k,0^n)}{V_{g,n+1}},\qquad
 \chi_1=0,\qquad \chi_k=\frac{\dot s_k}{p_k}\quad(k\geq2).
\]
Thus $\chi_k=-\dot w_k$ for every $k\geq1$.
The dilaton equation and differentiation of the finite polynomial
\eqref{eq:shifted} give
\[
 N\frac{V_{g,n}}{V_{g,n+1}}=\sum_{k\geq1}w_kr_k,
 \qquad
 \frac{\dd}{\dd t}\log V_{g,n}
 =\frac{V_{g,n+1}}{V_{g,n}}\sum_{k\geq1}\chi_kr_k.
\]
It follows that
\begin{equation}\label{eq:moment-quotient}
 \frac{\dd}{\dd t}\log V_{g,n}
 =N\frac{\sum_{k\geq1}\chi_kr_k}
 {\sum_{k\geq1}w_kr_k}.
\end{equation}
Since $n+1\geq3$, Proposition~\ref{prop:first-correction} applies.
For $k\geq1$, it gives
\begin{equation}\label{eq:inserted-deficit}
 r_k=1-\frac{E_k}{g}
 +O_n\!\left(\frac{(k+1)^4}{g^2}\right),\qquad
 E_k=2A\left[Bk^2+\bigl(B(n-3)+2M\bigr)k-(n-2)\right].
\end{equation}
Here and below the errors are uniform in $t$.
By \eqref{eq:weights}, the fourth absolute moments of $w_k$ and
$\chi_k$ are uniformly bounded. We may therefore sum
\eqref{eq:inserted-deficit} against either weight, with remainder
$O_n(g^{-2})$.

Differentiating \eqref{eq:S} gives
\begin{equation}\label{eq:differentiated-moments}
 \sum_{k\geq1}\chi_k=-\dot A,\qquad
 \sum_{k\geq1}k\chi_k=\frac{\dot A}{2},\qquad
 \sum_{k\geq1}k^2\chi_k=-\frac{\dot A}{4}+\frac{\dot\kappa}{16}.
\end{equation}
The termwise differentiations follow from the same uniform bounds.
Let
\[
 T_n=\sum_{k\geq1}w_kE_k,\qquad
 Q_n=\sum_{k\geq1}\chi_kE_k,\qquad
 \mathcal D_n=\frac{B(2n-7)}4+M+n-2.
\]
Substitution of \eqref{eq:S} and \eqref{eq:differentiated-moments}
into these sums yields
\begin{equation}\label{eq:moment-TQ}
 T_n=-2A^2\mathcal D_n-\frac A4,\qquad
 Q_n=2A\mathcal D_n\dot A+\frac{AB}{8}\dot\kappa.
\end{equation}
Since $B\kappa=2$, we obtain
\begin{equation}\label{eq:moment-cancellation}
 Q_n+\frac{\dot A}{A}T_n
 =-\frac{\dot A}{4}+\frac{A\dot\kappa}{4\kappa}.
\end{equation}
Thus \eqref{eq:moment-quotient} becomes
\[
 \frac{\dd}{\dd t}\log V_{g,n}
 =N\frac{-\dot A-Q_n/g+O_n(g^{-2})}
 {A-T_n/g+O_n(g^{-2})}.
\]
The denominator is bounded away from zero for all sufficiently large
$g$, uniformly in $t$, since $A\geq(4\pi^2)^{-1}$.
Using $N/g=2+O_n(g^{-1})$ and \eqref{eq:moment-cancellation}, we find
\begin{align*}
 \frac{\dd}{\dd t}\log V_{g,n}
 &=-N\frac{\dot A}{A}
 -\frac2A\left(Q_n+\frac{\dot A}{A}T_n\right)+O_n(g^{-1})\\
 &=-\left(2g+n-\frac52\right)\frac{\dot A}{A}
 -\frac12\frac{\dot\kappa}{\kappa}+O_n(g^{-1}).
\end{align*}
In the original notation, this is
\begin{equation}\label{eq:transport-proof}
 \frac{\dd}{\dd t}\log V^{(t)}_{g,n}
 =-(2g+n-5/2)\frac{\dot A_t}{A_t}
 -\frac12\frac{\dot\kappa_t}{\kappa_t}+O_n(g^{-1}).
\end{equation}
By \eqref{eq:amplitude}, it follows that
\[
 \frac{\dd}{\dd t}\log\mathcal B_{g,n}(t)=O_n(g^{-1})
\]
uniformly on $[0,1]$. Integration between any $u,t\in[0,1]$ gives
$\log(\mathcal B_{g,n}(t)/\mathcal B_{g,n}(u))=O_n(g^{-1})$.
Exponentiation proves \eqref{eq:transport-ratio}.
\end{proof}

\subsection{Variation with respect to a descendant time}
The same estimates give the derivative with respect to each time
$s_k$. For a parent with $n\geq2$ markings, define
\begin{equation}\label{eq:Rdef}
 R_{g,n,k}:=\partial_{s_k}\log V_{g,n}
 =\frac{J_{g,n+1}(k,0^n)}{p_k V_{g,n}},\qquad k\geq2.
\end{equation}
The derivative is taken with the other times fixed and is then
evaluated on the family \eqref{eq:family-y}.
The preceding calculation of $T_n$ gives
\begin{align}
 (2g+n-2)\frac{V_{g,n}}{V_{g,n+1}}
 &=A_t-\frac{T_n}{g}+O_n(g^{-2}),\label{eq:marking-next}\\
 T_n&=A_t^2\left[\frac{B_t(7-2n)}2-2M_t-2n+4\right]
 -\frac{A_t}4.\label{eq:T}
\end{align}
Inverting \eqref{eq:marking-next}, we obtain
\[
 \frac{V_{g,n+1}}{V_{g,n}}
 =\frac{2g+n-2}{A_t}+\frac{2T_n}{A_t^2}+O_n(g^{-1}).
\]
Multiply this expansion by \eqref{eq:E-expansion} at $n+1$ markings.
The result is
\begin{equation}\label{eq:R}
 R_{g,n,k}=\frac1{p_k}\left[
 \frac{2g+n-5/2}{A_t}
 -\frac{2(2k+1)}{\kappa_t}(2k+2n-4-\beta_t)
 \right]
 +\frac{O_n((k+1)^4/g)}{p_k}.
\end{equation}
This holds uniformly in $t$ and all $k\geq2$.
For $k\leq\sqrt g$, the errors in the product are bounded by
$C_n(k+1)^4/g$ before division by $p_k$.
For $k>\sqrt g$, monotonicity and the leading marking ratio give
$|p_kR_{g,n,k}|\leq C_ng$. The main term has magnitude at most
$C_n(g+k^2)\leq C_n'(k+1)^4/g$, which gives the same remainder
bound. At parent marking count $n=2$, the inserted descendant has
three markings, as required by Proposition~\ref{prop:first-correction}.

Differentiation along the family gives
\begin{equation}\label{eq:exact-derivative}
 \frac{\dd}{\dd t}\log V^{(t)}_{g,n}
 =\sum_{k\geq2}\dot s_k(t)R_{g,n,k}(t).
\end{equation}
By \eqref{eq:weights}, the sum of the remainder terms in
\eqref{eq:R} is $O_n(g^{-1})$, uniformly in $t$.

To identify the main term, allow an independent variation of $s_k$,
so that the zero $a$ of $y$ may vary. Differentiating $y(a)=0$,
the action integral, and $\kappa=-2a y'(a)$ gives
\begin{align}
 \partial_{s_k}A&=-\frac{2a^{2k+1}}{(2k+1)!!},\label{eq:variations}\\
 \partial_{s_k}a&=\frac{a^{2k-1}}{(2k-1)!!\,y'(a)},\notag\\
 \partial_{s_k}\kappa
 &=\frac{2a^{2k-1}}{(2k-1)!!}
 \left(2k-2-\frac{ay''(a)}{y'(a)}\right).\notag
\end{align}
At $a=1/2$, these formulas identify the main term in \eqref{eq:R} as
\begin{equation}\label{eq:model-derivative}
 \partial_{s_k}
 \left[-(2g+n-5/2)\log A+(n-1)\log(2a)
 -\frac12\log\kappa\right].
\end{equation}
In this substitution the conversion from $(2k-1)!!$ to $p_k$
contributes the factor $2(2k+1)$.
Along the family \eqref{eq:family-y}, $a=1/2$ is constant.
Summing \eqref{eq:model-derivative} against $\dot s_k$ therefore
gives \eqref{eq:transport-proof} again.

\begin{remark}
For the leading constant alone, one may replace the remainder in
Proposition~\ref{prop:first-correction} by $o(g^{-1})$ for each fixed
$k$, uniformly in $t$. The bound \eqref{eq:deficit} controls the
insertion remainder, after multiplication by $p_k$, by
$C_n(1+k^2)$. Equation~\eqref{eq:weights} then permits uniform
dominated convergence in \eqref{eq:exact-derivative}, giving
$\sup_t|\dd\log\mathcal B_{g,n}/\dd t|=o(1)$.
This already proves equality of the limiting amplitudes.
\end{remark}

\section{Painlev\'e I normalization}\label{sec:anchor}
At $t=0$, the spectral curve is $y_0=z-4z^3$. Thus $s_2=12$, and
all other shifted times vanish. In this section, $x$ denotes the
zero descendant time, rather than the spectral coordinate $x(z)$.
We use the genus convention $F=\sum_{g\geq0}\hbar^{2g}F_g$ for
the descendant free energy and put $u=F_{xx}$. The string equation and the KdV equation give
\begin{equation}\label{eq:cubic-string}
 u=x+6u^2+\hbar^2u'',\qquad
 u(x,\hbar)=\sum_{g\geq0}\hbar^{2g}u_g(x).
\end{equation}
Indeed, at $s_2=12$ the string equation is
$F_x=x^2/2+12F_{s_1}$. Differentiating once and substituting
$F_{xs_1}=u^2/2+\hbar^2u''/12$ gives
\eqref{eq:cubic-string}.

Put $X=1-24x$. The formal solution with the topological initial
condition has the form
\begin{equation}\label{eq:cubic-u}
 u_0=\frac{1-\sqrt X}{12},\qquad
 u_g=c_gX^{(1-5g)/2}\quad(g\geq1).
\end{equation}
Substitution into \eqref{eq:cubic-string} gives
\begin{equation}\label{eq:c-rec}
 c_1=12,\qquad
 c_g=6\sum_{h=1}^{g-1}c_hc_{g-h}
 +144(5g-6)(5g-4)c_{g-1}\quad(g\geq2).
\end{equation}
Since differentiation twice with respect to $x$ inserts two zero
descendants, we have
\begin{equation}\label{eq:c-volume}
 V^{(0)}_{g,2}=u_g(0)=c_g.
\end{equation}
Further differentiation of \eqref{eq:cubic-u} gives, for $n\geq2$,
\[
 V^{(0)}_{g,n}
 =c_g12^{n-2}\prod_{j=0}^{n-3}(5g-1+2j).
\]
To obtain the one-marking value, we use the string equation at zero
times in the form $F_{00}=12F_{01}$. The dimension constraint and the
dilaton equation at $s_2=12$ give $F_{01}=(5g-3)F_0$. Hence
$V^{(0)}_{g,1}=c_g/[12(5g-3)]$, which proves
\eqref{eq:cubic-fixedn} without an integration constant.
For $g\geq2$, a further application of the string and dilaton
equations gives
$V^{(0)}_{g,0}=c_g/[144(5g-3)(5g-5)]$.
Here the unmarked value is the connected shifted descendant free
energy.

Let $v_g$ be the coefficients of the standard Painlev\'e~I expansion
\begin{equation}\label{eq:PI}
 -\frac16v''+v^2=z,\qquad
 v(z)=z^{1/2}\sum_{g\geq0}v_g z^{-5g/2},\qquad v_0=1.
\end{equation}
They satisfy
\begin{equation}\label{eq:PI-rec}
 v_g=\frac{25(g-1)^2-1}{48}v_{g-1}
 -\frac12\sum_{h=1}^{g-1}v_hv_{g-h},\qquad v_1=-\frac1{48}.
\end{equation}
Comparison of \eqref{eq:c-rec} and \eqref{eq:PI-rec}, including
their initial values, yields
\begin{equation}\label{eq:PI-scale}
 c_g=-\frac{6912^g}{12}\,v_g.
\end{equation}
The linear terms agree because $6912/48=144$ and
$(5g-6)(5g-4)=25(g-1)^2-1$; the quadratic terms agree under the
same substitution.

The large-order theorem of \cite[Section 2.3, equation (2.39)]{GIKM}
states that
\begin{equation}\label{eq:PI-asym}
 v_g=-\frac{3^{1/4}}{2\pi^{3/2}}
 \left(\frac{8\sqrt3}{5}\right)^{-2g+1/2}
 \Gamma(2g-1/2)\bigl(1+O(g^{-1})\bigr).
\end{equation}
This is the proved zero-instanton coefficient asymptotic, obtained
from the Riemann--Hilbert analysis of tritronqu\'ee solutions.
We use no assertion about higher-instanton resurgence.

Writing $A_{\rm PI}=8\sqrt3/5$, we have
\[
 6912/A_{\rm PI}^2=900=30^2,
 \qquad 3^{1/4}A_{\rm PI}^{1/2}=2\sqrt{6/5}.
\]
It follows from \eqref{eq:PI-scale}--\eqref{eq:PI-asym} that
\begin{equation}\label{eq:c-asym}
 V^{(0)}_{g,2}
 =\frac{30^{2g-1/2}\Gamma(2g-1/2)}{2\pi^{3/2}}
 \bigl(1+O(g^{-1})\bigr).
\end{equation}
Since $A_0=1/30$ and $\kappa_0=2$, this gives
\begin{equation}\label{eq:anchor-amplitude}
 \mathcal B_{g,2}(0)
 =\frac1{\sqrt2\,\pi^{3/2}}\bigl(1+O(g^{-1})\bigr).
\end{equation}

\section{The Weil--Petersson constant and fixed marking numbers}
\label{sec:constant}
\begin{proof}[Proof of Corollary~\ref{cor:constant}]
Put $\mathcal B_g(t)=\mathcal B_{g,2}(t)$. By
\eqref{eq:uniform-transport}, there is a constant $C$, independent
of $g$ and $t$, such that, for all sufficiently large $g$,
\[
 \left|\frac{\dd}{\dd t}\log\mathcal B_g(t)\right|
 \leq\frac Cg,\qquad 0\leq t\leq1.
\]
Integrating over $[0,1]$, we obtain
\[
 \left|\log\frac{\mathcal B_g(1)}{\mathcal B_g(0)}\right|
 \leq\frac Cg.
\]
Thus
\[
 e^{-C/g}\leq\frac{\mathcal B_g(1)}{\mathcal B_g(0)}
 \leq e^{C/g},\qquad
 \frac{\mathcal B_g(1)}{\mathcal B_g(0)}=1+O(g^{-1}).
\]
Together with \eqref{eq:anchor-amplitude}, this gives
\[
 \mathcal B_g(1)=\frac1{\sqrt2\,\pi^{3/2}}
 \bigl(1+O(g^{-1})\bigr).
\]
Since $A_1=1/(4\pi^2)$ and $\kappa_1=1$, the definition of
$\mathcal B_g$ now yields
\begin{equation}\label{eq:wp-new}
 V^{\WP}_{g,2}
 =\frac{(4\pi^2)^{2g-1/2}\Gamma(2g-1/2)}
 {\sqrt2\,\pi^{3/2}}
 \bigl(1+O(g^{-1})\bigr).
\end{equation}
The Gamma quotient satisfies
\[
 \frac{\Gamma(2g-1/2)}{\Gamma(2g)}
 =(2g)^{-1/2}\bigl(1+O(g^{-1})\bigr).
\]
Consequently, \eqref{eq:wp-new} takes the form
\[
 V^{\WP}_{g,2}
 =\frac1{\sqrt\pi}
 \frac{(4\pi^2)^{2g-1}\Gamma(2g)}{\sqrt g}
 \bigl(1+O(g^{-1})\bigr).
\]
In the normalization of \eqref{eq:mz}, this is
\begin{align*}
 C_{\WP}
 &=\lim_{g\to\infty}
 \frac{(4\pi^2)^{1/2}}{\sqrt2\,\pi^{3/2}}
 \frac{\sqrt g\,\Gamma(2g-1/2)}{\Gamma(2g)}\\
 &=\frac{2\pi}{\sqrt2\,\pi^{3/2}}\cdot\frac1{\sqrt2}
 =\frac1{\sqrt\pi}.
\end{align*}
\end{proof}

We next extend the asymptotic to every fixed number of markings.

\begin{proposition}[Asymptotics for fixed marking number]
\label{prop:all-n}
For each fixed $n\geq0$, uniformly for $t\in[0,1]$,
\begin{equation}\label{eq:all-n-asym}
 V^{(t)}_{g,n}=
 \frac{\Gamma(2g+n-5/2)A_t^{-2g-n+5/2}}
 {\sqrt2\,\pi^{3/2}\sqrt{\kappa_t}}
 \bigl(1+O_n(g^{-1})\bigr).
\end{equation}
For $n=0$, the left side is defined by the connected shifted
descendant free energy
\[
 V^{(t)}_{g,0}:=
 \left\langle\exp\left(\sum_{k\geq2}s_k(t)\tau_k\right)\right\rangle_g,
 \qquad g\geq2.
\]
At $t=1$, we have $V^{(1)}_{g,0}=V^{\WP}_{g,0}$ exactly.
For general $t$, the unmarked positivity asserted here holds for
all sufficiently large $g$, with a threshold independent of $t$.
The comparison \eqref{eq:endpoint-upper} remains restricted to
$n\geq1$.
\end{proposition}
\begin{proof}
For $n\geq2$, the exact cubic fixed-marking formula and
\eqref{eq:c-asym} give the normalized leading amplitude
$1/(\sqrt2\pi^{3/2})$. Uniform transport proves
\eqref{eq:all-n-asym} for these marking numbers. The leading
marking ratio with one parent marking then gives the result for
$n=1$.

For $n=0$, we first extend the coarse descendant deficit estimate
to one marking, using the estimates already proved. We have
\[
 \frac{V_{g-1,3}}{V_{g,1}}
 =\frac{V_{g-1,4}}{V_{g,2}}
 \frac{V_{g,2}/V_{g,1}}{V_{g-1,4}/V_{g-1,3}}
 =1+O(g^{-1}).
\]
Here \eqref{eq:genus-leading} applies with two parent markings,
and both marking ratios have leading term $(2g-1)/A_t$.
The lower marking estimate now gives
$V_{g-1,2}=O(V_{g,1}/g)$.

With one parent marking, the only separating marking pattern is
$(1,1)$. For $1\leq h\leq g/2$, the refined genus comparison with
two reference markings gives
\[
 \frac{V_{h,1}V_{g-h,1}}{V_{g,1}}
 \leq C h^P\frac{r_*^{2h}}{\binom gh^2}.
\]
Indeed, conversion of the large factor and the parent to two
markings contributes at most $C g/(g-h)$, which is bounded in this
range. Genus-zero factors with one marking are unstable. The
convolution lemma therefore bounds the separating contribution by
$O(V_{g,1}/g^2)$.

The exact one-marking defect has no merge terms. Applying the
preceding bounds and descendant monotonicity to its nonseparating
and separating terms, we obtain
\[
 0\leq1-J_{g,1}(k)/V_{g,1}\leq\min\{1,Ck^2/g\}.
\]
The exact shifted dilaton identity with no parent markings then
gives
\[
 (2g-2)\frac{V_{g,0}}{V_{g,1}}=A_t+O(g^{-1}).
\]
Since $A_t\geq1/(4\pi^2)$ and the error is uniform in $t$, there
is a single $g_0$ such that $V^{(t)}_{g,0}>0$ for all
$g\geq g_0$ and $t\in[0,1]$. This proves
\eqref{eq:all-n-asym} for $n=0$. It gives no all-genus comparison
between the unmarked values and their cubic and WP endpoints.

For unmarked curves, the sine specialization is again
$\exp(2\pi^2\kappa_1)$, with Schur series
$1-\sum_{k\geq2}s_kv^{k-1}=e^{-2\pi^2v}$.
Thus the endpoint $t=1$ has the ordinary WP normalization.
\end{proof}

\begin{remark}[Unmarked positivity]
The polynomial-positivity argument for
\eqref{eq:endpoint-upper} does not apply to the free energy.
For a general kernel
$q(s)=1+b_1s+b_2s^2+b_3s^3+\cdots$, the genus-two unmarked value is
\[
 V_{2,0}=-\frac4{15}b_1^3+\frac{17}{6}b_1b_2+\frac{35}{6}b_3.
\]
To obtain this identity, expand
$s_2^3\langle\tau_2^3\rangle_2/6+
s_2s_3\langle\tau_2\tau_3\rangle_2+s_4\langle\tau_4\rangle_2$,
using the intersection numbers $7/240$, $29/5760$, and $1/1152$,
respectively. The resulting polynomial has a negative coefficient
in the kernel variables. Thus the marked induction does not prove
an unmarked comparison, although such a comparison might still
hold for the family considered here.
At the endpoints the formula gives $V^{(0)}_{2,0}=42/5$ and
$V^{(1)}_{2,0}=43\pi^6/2160$, in agreement with the normalizations
above. The determination of the two-marking constant requires no
unmarked positivity assumption.
\end{remark}

The proof of the refined splitting estimate uses the one-sided
genus inequality, not the genus asymptotic obtained from it.
The bound $B_tA_t/A_0\leq15/\pi^2<2$ permits comparison of the
smaller component with the cubic model and of the larger component
with the deformed action. The Painlev\'e coefficient theorem enters
only after uniform transport has been proved, to determine the
absolute normalization.

\appendix
\section{The first correction term}
The published WP correction coefficient provides a comparison with
the cubic model. It is not used in the proof of the leading constant.
Write
\[
 \mathcal B_{g,n}(j)=\mathcal B_\infty
 \left(1+\frac{\beta_n^{(j)}}g+O_n(g^{-2})\right),\qquad
 j=0,1,\qquad \mathcal B_\infty=\frac1{\sqrt2\pi^{3/2}}.
\]
The all-orders zero-instanton expansion proved in
\cite[(2.43)--(2.45)]{GIKM}, with the first coefficient given by (1.7),
is
\[
 v_g=-\frac{3^{1/4}}{2\pi^{3/2}}A_{\rm PI}^{-2g+1/2}
 \Gamma(2g-1/2)
 \left(1-\frac1{8(2g-3/2)}+O(g^{-2})\right).
\]
The cubic rescaling and the formula for a fixed number of markings give
\begin{equation}\label{eq:beta-cubic}
 \beta_n^{(0)}=-\frac1{16}-\frac{(n-2)(n-4)}{20}.
\end{equation}
For $n\geq2$, expand the finite product in
\eqref{eq:cubic-fixedn} and the corresponding Gamma ratio. For $n=1$
and $n=0$, the same formula follows from the exact cubic relations.

The WP coefficient in \cite[Remark 1.3]{MZ} is
\[
 c_n^{(1)}=-\frac{n^2}{2\pi^2}
 -\left(\frac14-\frac5{2\pi^2}\right)n
 +\frac7{12}-\frac{17}{6\pi^2}.
\]
Its derivation does not require the value of the leading constant.
Since
\[
 \frac{\sqrt{2g}\,\Gamma(2g+n-5/2)}{\Gamma(2g+n-2)}
 =1+\frac{11-4n}{16g}+O_n(g^{-2}),
\]
the normalization used here gives
\begin{equation}\label{eq:beta-WP}
 \beta_n^{(1)}=-\frac5{48}
 +\frac{-3n^2+15n-17}{6\pi^2}.
\end{equation}
In particular,
\begin{equation}\label{eq:nonzero-correction}
 \frac{\mathcal B_{g,2}(1)}{\mathcal B_{g,2}(0)}
 =1+\frac{-1/24+1/(6\pi^2)}g+O(g^{-2}).
\end{equation}
The coefficient is approximately $-0.0247798027$. Thus the limiting
amplitude is constant along the path, but the first correction term is
not. Equation~\eqref{eq:nonzero-correction} gives the coefficient of
$g^{-1}$ in the endpoint ratio left undetermined by the uniform
transport theorem.

\section{Transport equation for the first correction}\label{app:next}
We compute the first correction term for every $t\in[0,1]$. The proof
uses the recursion and uniform estimates established above. Throughout
this appendix, write $A=A_t$, $B=B_t$, and $M=M_t$.

\subsection{Several positive descendants and the child ratios}
\begin{lemma}[Several descendants]\label{lem:multi-descendants}
Let $N\geq3$ be fixed. For a nonnegative degree vector $\mathbf d$, put
$d_*=|\mathbf d|$ and $m=\#\{i:d_i>0\}$, and define
\begin{equation}\label{eq:multi-E}
 \mathcal E_N(\mathbf d)=2A\left[
 Bd_*^2+\bigl(B(N-4)+2M\bigr)d_*+\frac{m(m+5-2N)}2\right].
\end{equation}
Then, uniformly in $t\in[0,1]$ and in $\mathbf d$,
\begin{equation}\label{eq:multi-expansion}
 \frac{J_{g,N}(\mathbf d)}{V_{g,N}}
 =1-\frac{\mathcal E_N(\mathbf d)}g
 +O_N\left(\frac{(d_*+1)^4}{g^2}\right).
\end{equation}
In particular, $\mathcal E_N(\bzero)=0$, and the case of one positive
degree agrees with \eqref{eq:E}.
\end{lemma}
\begin{proof}
By symmetry we may increase the first degree. Suppose first that
$d_1\geq1$. The merge terms in the adjacent-difference recursion have
total weight
$B\sum_{j=2}^N(2d_j+1)=B(2d_*-2d_1+N-1)$.
The number of nonnegative pairs $u+v=L+d_1-2$ is $L+d_1-1$;
this formula also gives zero when $d_1=1$ and $L=0$.
Its weighted sum is $M+(d_1-1)B$.
Using the two leading child ratios \eqref{eq:child-leading}, the
coefficient of $g^{-1}$ is therefore
\[
 2AB(2d_*-2d_1+N-1)+4A\{M+(d_1-1)B\}
 =2A\{B(2d_*+N-3)+2M\}.
\]

If $d_1=0$, there are $N-1-m$ other zero degrees. Each such merge
loses its $L=0$ term, whose weight is $\Delta_0=1$.
The nonseparating pair count is now $(L-1)_+$, with weighted sum
$M-B+1$. The merge loss is $-2A(N-1-m)$ and the change from the
generic pair count is $4A$. Thus in all cases the leading increment is
\begin{equation}\label{eq:multi-increment}
 I_i(\mathbf d)=2A\left\{B(2d_*+N-3)+2M
 -(N-3-m)\boldsymbol1_{\{d_i=0\}}\right\}.
\end{equation}
Direct subtraction in \eqref{eq:multi-E} gives
$I_i(\mathbf d)=\mathcal E_N(\mathbf d+\mathbf e_i)
-\mathcal E_N(\mathbf d)$: when $d_i=0$, the change in the last
term inside the brackets is $m+3-N$.

We next bound the error. Assume $d_*\leq\sqrt g$.
Each child degree has sum at most $d_*+L+1$.
The children in the merge and nonseparating terms have respectively
$N-1\geq2$ and $N+1$ markings. Proposition~\ref{prop:deficit}
therefore bounds the error in replacing a child by its cusp value by
$C_N(d_*+L+1)^2/g$. Its cusp ratio is $O_N(g^{-1})$.
The merge multiplicities sum to $O_N(d_*+1)$, and the number of
node pairs is $O(d_*+L+1)$. The error in either leading cusp ratio
is $O_N(g^{-2})$. By \eqref{eq:split-fine}, the separating terms
contribute at most $C_N(d_*+1)/g^2$ after the kernel sum.
Consequently the normalized adjacent error is bounded by
\[
 \frac{C_N}{g^2}\sum_{L\geq0}\Delta_L(d_*+L+1)^3
 \leq\frac{C_N'(d_*+1)^3}{g^2}.
\]
Extending the original finite $L$-sum to infinity changes this bound
by an exponentially small quantity: its cutoff is
$3g-3+N-d_*\geq c g$ for large $g$, and
Lemma~\ref{lem:kernel} applies.

Sum the adjacent estimate along any path of $d_*$ unit increments
from $\bzero$ to $\mathbf d$. The leading terms telescope to
$\mathcal E_N(\mathbf d)$, while the errors have sum at most
$C_N(d_*+1)^4/g^2$. This proves \eqref{eq:multi-expansion}
when $d_*\leq\sqrt g$.
For $d_*>\sqrt g$, monotonicity and
$|\mathcal E_N(\mathbf d)|\leq C_N(d_*+1)^2$ give
\[
 \left|\frac{J_{g,N}(\mathbf d)}{V_{g,N}}-1
       +\frac{\mathcal E_N(\mathbf d)}g\right|
 \leq 1+\frac{C_N(d_*+1)^2}{g}
 \leq\frac{C_N'(d_*+1)^4}{g^2}.
\]
This includes degrees beyond the dimension cutoff.
\end{proof}

Let $T_n$ be as in \eqref{eq:T}, and set
\begin{align}
 \alpha_n&=\frac{n-2}{2}+\frac{T_n}{A}
 =\frac{2n-5}{4}
 +A\left[\frac{B(7-2n)}2-2M-2n+4\right],\label{eq:alpha}\\
 \gamma_n&=-\mathcal E_n(1,0^{n-1})-6A(n-2)
 =-2A\bigl[(B+2)n-3B-3+2M\bigr].\label{eq:gamma}
\end{align}
The marking expansion gives
\[
 \frac{V_{g,n+1}}{V_{g,n}}
 =\frac{2g}{A}\left(1+\frac{\alpha_n}g+O_n(g^{-2})\right),
 \qquad n\geq2.
\]
For $n\geq3$, the KdV equation and \eqref{eq:multi-expansion} give
\[
 \frac{V_{g-1,n+2}}{V_{g,n}}
 =1+\frac{\gamma_n}g+O_n(g^{-2}).
\]
The term $6A(n-2)$ comes from the leading product contribution. There
are two orientations of a genus-zero component with three markings,
each with binomial multiplicity $n-2$. All positive-genus products
and the remaining genus-zero products are $O(g^{-2})$ relative to
$V_{g,n}$. The stable genus-zero terms must therefore be retained at
this order.

For $N\geq4$, put $\eta_N=1+\gamma_N-\alpha_{N+1}$. The required
child ratios are
\begin{align}
 \frac{V_{g,N-1}}{V_{g,N}}
 &=\frac{A}{2g}\left(1-\frac{\alpha_{N-1}}g+O_N(g^{-2})\right),
 \label{eq:next-child}\\
 \frac{V_{g-1,N+1}}{V_{g,N}}
 &=\frac{A}{2g}\left(1+\frac{\eta_N}g+O_N(g^{-2})\right).
 \notag
\end{align}

\subsection{The second descendant coefficient}
Write $b_1=[s]q_t(s)=1+t(\zeta(2)-1)$. For $N\geq4$ and $d\geq0$,
define
\begin{align}
 \mathcal C_{N,d}={}&
 2A(N-1)\sum_{\substack{L\geq0\\L+d\geq1}}\Delta_L
 \left[-\alpha_{N-1}-\mathcal E_{N-1}(L+d-1,0^{N-2})\right]
 \label{eq:Cnext}\\
 &+4A\sum_{L\geq0}\Delta_L\sum_{u+v=L+d-2}
 \left[\eta_N-\mathcal E_{N+1}(u,v,0^{N-1})\right]\notag\\
 &+4A^2\left[
 \left(\binom{N-1}{2}+\frac{b_1}{2}\right)
 \sum_{\substack{L\geq0\\L+d\geq2}}\Delta_L
 +\frac12\sum_{\substack{L\geq0\\L+d\geq3}}\Delta_L\right].\notag
\end{align}
These sums converge uniformly. A pair sum with negative total is
empty. Set
\begin{equation}\label{eq:Qnext}
 Q_{N,0}=0,\qquad Q_{N,k}=-\sum_{d=0}^{k-1}\mathcal C_{N,d}.
\end{equation}

\begin{proposition}[Second descendant correction]\label{prop:second}
For fixed $N\geq4$ and all $k\geq0$, uniformly in $t$,
\begin{equation}\label{eq:second-expansion}
 \frac{J_{g,N}(k,0^{N-1})}{V_{g,N}}
 =1-\frac{E_{N,k}}g+\frac{Q_{N,k}}{g^2}
 +O_N((k+1)^6/g^3),
\end{equation}
where $E_{N,0}=0$.
\end{proposition}
\begin{proof}
Put $I_d=E_{N,d+1}-E_{N,d}$. We first prove, for $d\leq\sqrt g$,
\begin{equation}\label{eq:second-adjacent}
 \frac{J_{g,N}(d,0^{N-1})-J_{g,N}(d+1,0^{N-1})}{V_{g,N}}
 =\frac{I_d}{g}+\frac{\mathcal C_{N,d}}{g^2}
 +O_N\left(\frac{(d+1)^5}{g^3}\right).
\end{equation}
For a merge child of degree vector $\mathbf e$, multiplication of
\eqref{eq:multi-expansion} by \eqref{eq:next-child} gives
\[
 \frac{J_{g,N-1}(\mathbf e)}{V_{g,N}}
 =\frac A{2g}\left[1+
 \frac{-\alpha_{N-1}-\mathcal E_{N-1}(\mathbf e)}g
 +O_N\left(\frac{(|\mathbf e|+1)^4}{g^2}\right)\right].
\]
For the nonseparating child the coefficient in brackets is
$\eta_N-\mathcal E_{N+1}(\mathbf e)$ instead. Replacing $g-1$ by
$g$ in the child expansion changes only the stated error. Thus these
two contributions give the first two lines of \eqref{eq:Cnext}.

The separating contributions of order $g^{-2}$ are as follows. The
column headed multiplicity counts one orientation; each row occurs
in both orientations.
\[
\begin{array}{c|c|c|c|c}
 (h,k)&\text{node degree}&\text{small coefficient}
 &\text{multiplicity}&\text{condition}\\\hline
 (0,3)&0&1&\binom{N-1}{2}&L+d\geq2\\
 (1,1)&0&b_1/2&1&L+d\geq2\\
 (1,1)&1&1/2&1&L+d\geq3
\end{array}
\]
In the first row, the large cusp factor is $V_{g,N-2}$; in the last
two rows it is $V_{g-1,N}$. In each case its ratio to $V_{g,N}$ is
$A^2/(4g^2)(1+O_N(g^{-1}))$, by the marking and genus ratios.
The recursion coefficient $8$, the two orientations, and these
ratios give exactly the last line of \eqref{eq:Cnext}.
For each listed node, the other node degree is fixed by
$u+v=L+d-2$; hence there is no additional pair multiplicity.

We bound all remaining splitting types after division by $V_{g,N}$.
For a fixed small component with $h\geq1$ and $k$ markings, its
cusp contribution has order $g^{1-k-2h}$. Thus $h=1$, $k\geq2$
gives $O_N(g^{-3})$. For $h\geq2$, the proof of
Lemma~\ref{lem:binomial}, with the $h=1$ term omitted, gives
\[
 \sum_{h=2}^{\lfloor g/2\rfloor}
 \frac{h^P r_*^{2h}}{\binom gh^2}=O(g^{-4}).
\]
Indeed, each fixed term is $O(g^{-2h})$, and the two remaining
ranges in that proof are smaller than $g^{-4}$ after choosing its
fixed cutoff sufficiently large. The marking factor in
\eqref{eq:product-fine} is $g^{1-k}\leq1$.
A genus-zero small component with $k\geq4$ gives
$O_N(g^{1-k})=O_N(g^{-3})$. Interchanging components covers both
halves of the genus sum. Monotonicity bounds all their descendants
by cusp coefficients, and node counting costs at most
$C_N(d+L+1)$. These observations give
$C_N(d+1)/g^3$ for the remaining splitting terms after summing $L$.

For the three leading rows, the large child has at least two
markings. Its descendant deficit is at most
$C_N(d+L+1)^2/g$, and its cusp-ratio error is $O_N(g^{-3})$.
The resulting error after summing $L$ is
$O_N((d+1)^2/g^3)$. Finally, the merge and nonseparating remainders
above have bound $C_N(d+L+1)^4/g^3$ per node pair. Counting the
pairs and using the kernel moments bounds the total by
\[
 \frac{C_N}{g^3}\sum_{L\geq0}\Delta_L(d+L+1)^5
 \leq\frac{C_N'(d+1)^5}{g^3}.
\]
The dimension-cutoff tails are exponentially small for
$d\leq\sqrt g$. This proves \eqref{eq:second-adjacent}.

Sum \eqref{eq:second-adjacent} over $0\leq d<k$.
The two coefficients telescope to $E_{N,k}$ and $-Q_{N,k}$, and
the errors sum to $O_N((k+1)^6/g^3)$ for $k\leq\sqrt g$.
Equation~\eqref{eq:Cnext} also gives
$|\mathcal C_{N,d}|\leq C_N(d+1)^3$ and hence
$|Q_{N,k}|\leq C_N(k+1)^4$. When $k>\sqrt g$, monotonicity bounds
the error in \eqref{eq:second-expansion} by
\[
 C_N\left(1+\frac{k^2}{g}+\frac{k^4}{g^2}\right)
 \leq C_N'\frac{(k+1)^6}{g^3}.
\]
The result follows for all $k$.
\end{proof}

For comparison, the exact cubic formula is
\[
 \frac{J^{(0)}_{g,N}(2,0^{N-1})}{V^{(0)}_{g,N}}
 =1-\frac{N-1}{3(5g+2N-7)},\qquad
 Q_{N,2}(0)=\frac{(N-1)(2N-7)}{75}.
\]
The value of $Q_{N,2}(0)$ also follows from \eqref{eq:Cnext}, since
$\Delta_0=1$ and $\Delta_L=0$ for $L>0$.

\subsection{The uniform forcing term and its integral}
For a parent with $n\geq3$ markings, define
\begin{align}
 U_n&=\sum_{k\geq1}w_kQ_{n+1,k},\label{eq:Unext}\\
 \delta_n&=(T_n/A)^2-U_n/A+\frac{(n-2)T_n}{2A},\notag\\
 \mathfrak F_n(t)&=\frac2A\sum_{k\geq2}\frac{\dot s_k}{p_k}
 \left[Q_{n+1,k}-\alpha_nE_{n+1,k}+\delta_n\right].
 \label{eq:forcing}
\end{align}
The polynomial bounds above and \eqref{eq:weights} imply uniform
absolute convergence. Sum \eqref{eq:second-expansion} against the
dilaton weights and invert to obtain
\[
 \frac{V_{g,n+1}}{V_{g,n}}
 =\frac{2g}{A}
 \left(1+\frac{\alpha_n}g+\frac{\delta_n}{g^2}
 +O_n(g^{-3})\right).
\]
Multiplication by the descendant expansion gives the insertion ratio.
After contraction against $\dot s_k$, the leading and constant terms
are those in \eqref{eq:transport-proof}, and the next term gives
\begin{equation}\label{eq:forcing-equation}
 \frac{\dd}{\dd t}\log\mathcal B_{g,n}(t)
 =\frac{\mathfrak F_n(t)}g+O_n(g^{-2}),\qquad n\geq3,
\end{equation}
uniformly on $[0,1]$. The error after contraction is $O(g^{-2})$:
before contraction, the remainder has a sixth-degree polynomial bound.
For $k>\sqrt g$, its magnitude before division by $p_k$ is bounded
directly by $C_n(g+k^2+k^4/g)\leq C_n'(k+1)^6/g^2$.
This direct bound applies in the range where multiplication of the
descendant remainder would give a weaker estimate.

Integration, with the initial value supplied by the cubic model, gives
\begin{equation}\label{eq:b-path}
 \mathcal B_{g,n}(t)=\frac1{\sqrt2\pi^{3/2}}
 \left(1+\frac{\mathfrak b_n(t)}g+O_n(g^{-2})\right),\qquad
 \mathfrak b_n(t)=\beta_n^{(0)}+\int_0^t\mathfrak F_n(v)\,\dd v,
 \quad n\geq3.
\end{equation}
The marking ratio gives the corresponding result for two markings:
\[
 \frac{\mathcal B_{g,3}(t)}{\mathcal B_{g,2}(t)}
 =1+\frac{\alpha_2+1/4}{g}+O(g^{-2})
 =1+\frac{3A_t2^t(1-t)}{2g}+O(g^{-2}).
\]
Hence
\begin{equation}\label{eq:b2-path}
 \mathfrak b_2(t)=-\frac1{80}
 +\int_0^t\mathfrak F_3(v)\,\dd v
 -\frac32 A_t2^t(1-t).
\end{equation}
Only division of uniform expansions is used here; no asymptotic
remainder in the marking relation is differentiated.

Comparison with the endpoint coefficient \eqref{eq:beta-WP} gives
\begin{equation}\label{eq:forcing-integral}
 \int_0^1\mathfrak F_3(t)\,\dd t
 =-\frac{11}{120}+\frac1{6\pi^2}.
\end{equation}
The next subsection evaluates this integral directly from
\eqref{eq:forcing}, without using the published WP correction.

\subsection{A closed formula for the first correction}
Put $\mu_j=\sum_{L\geq0}L^j\Delta_L$, so that $\mu_0=B$ and
$\mu_1=M$. We shall prove
\begin{equation}\label{eq:primitive}
 \mathfrak F_3(t)=\frac{\dd}{\dd t}
 \left\{A\left(\frac13-b_1-\frac32\mu_2+\frac43M
 +\frac7{24}B-\frac{M^2}{3B}\right)\right\}.
\end{equation}
We first evaluate \eqref{eq:Cnext} at $N=4$. Put
$r=L+d$, $\alpha=\alpha_3$, and $\eta=\eta_4$. From
\eqref{eq:alpha}--\eqref{eq:gamma},
\[
 \alpha=\frac14+A\left(\frac B2-2M-2\right),\qquad
 \eta=-\frac14+A\left(-\frac B2-2M-4\right).
\]
For $r\geq3$, the merge child has positive degree $r-1$, and
\[
 \mathcal E_3(r-1)=2A\{B(r-1)^2+(-B+2M)(r-1)\}.
\]
In the nonseparating sum put $j=r-2\geq1$. Of the $j+1$ pairs
$u+v=j$, two have one positive coordinate and $j-1$ have two.
The last term in \eqref{eq:multi-E} is respectively $-2$ and $-3$
for these pairs when $N=5$. It follows that
\[
 \sum_{u+v=r-2}\mathcal E_5(u,v)
 =2A\left[(r-1)\{B(r-2)^2+(B+2M)(r-2)\}-3r+5\right].
\]
Thus the summand in \eqref{eq:Cnext}, before multiplication by
$\Delta_L$, is the following cubic polynomial for $r\geq3$:
\begin{align}
 C(r)={}&-8A^2Br^3+(20A^2B-16A^2M)r^2\notag\\
 &+(-6A^2B+16A^2M+8A^2-A)r\label{eq:C-polynomial}\\
 &-9A^2B+12A^2M+2A^2b_1+2A^2-A/2.\notag
\end{align}
There are three corrections to this polynomial. At $r=0$ every
sum is empty. At $r=1$ only the merge term survives, with child
degree zero. At $r=2$ the nonseparating child is $(0,0)$ and the
degree-one genus-one node is absent. Subtraction of
\eqref{eq:C-polynomial} gives
\[
\begin{array}{c|c|c}
 r&\text{actual summand}&\text{correction }e_r\\\hline
 0&0&-C(0)\\
 1&-6A\alpha&2A^2(1-b_1)\\
 2&6A[-\alpha-\mathcal E_3(1)]+4A\eta
       +4A^2(3+b_1/2)&-10A^2
\end{array}
\]
Set $H_*(d)=\sum_{L\geq0}\Delta_L C(d+L)=\sum_{j=0}^3h_jd^j$.
The moment identities give
\begin{align*}
 h_3={}&-8A^2B^2,\\
 h_2={}&20A^2B^2-40A^2BM,\\
 h_1={}&-AB+A^2(-6B^2+56BM-24B\mu_2+8B-32M^2).
\end{align*}
For $k\geq3$, all three exceptional values have occurred in the
sum over $0\leq d<k$. Since $\Delta_0=1$,
$\Delta_0+\Delta_1=b_1$, and
$\Delta_0+\Delta_1+\Delta_2=b_2$, the polynomial part is
\[
 P(k)=-\sum_{d=0}^{k-1}H_*(d)-e_0-b_1e_1-b_2e_2.
\]
Here the sum of the cubic polynomial is interpreted by the usual
formulas for sums of powers, so $P$ is a polynomial in $k$ of degree
at most four. At $k=1$ the missing corrections are $e_1+b_1e_2$;
at $k=2$ the missing correction is $e_2$. Therefore
\begin{equation}\label{eq:quartic-exceptions}
 Q_{4,k}=P(k)+(2-12b_1)A^2\boldsymbol1_{k=1}
 -10A^2\boldsymbol1_{k=2},\qquad k\geq1.
\end{equation}
Writing $P(k)=\sum_{j=0}^4q_jk^j$, the formulas
$\sum_{d<k}d=(k^2-k)/2$,
$\sum_{d<k}d^2=(2k^3-3k^2+k)/6$, and
$\sum_{d<k}d^3=(k^4-2k^3+k^2)/4$ yield
\[
 q_4=-h_3/4,\qquad q_3=h_3/2-h_2/3,\qquad
 q_2=-h_3/4+h_2/2-h_1/2.
\]
The coefficients of degrees two, three, and four are
\begin{align}
 q_2&=\frac{AB}{2}
 +A^2(16M^2-4B+12B\mu_2-48BM+15B^2),\label{eq:q234}\\
 q_3&=\frac{A^2}{3}(40BM-32B^2),\qquad
 q_4=2A^2B^2.\notag
\end{align}
The constant and linear coefficients cancel from the forcing term.
To verify this and compute the antiderivative, set
$S_j=\sum_{k\geq1}k^jw_k$. The Euler-operator calculation at the
root used in \eqref{eq:S} gives
\begin{align}
 S_0&=A,&S_1&=-A/2,\notag\\
 S_2&=A/4+r_2,& r_2&=-\frac1{8B},\label{eq:Snext}\\
 S_3&=-A/8+r_3,& r_3&=-\frac5{16B}+\frac{M}{4B^2},\notag\\
 S_4&=A/16+r_4,& r_4&=-\frac{23}{32B}+\frac{M}{B^2}
 +\frac{3\mu_2}{8B^2}-\frac{3M^2}{4B^3}.\notag
\end{align}
To compute these moments, keep the curve fixed and replace the
upper endpoint of the action integral by $a$. The coefficient of
$a^{2k+1}$ is an eigenfunction of $(a\partial_a-1)/2$ with
eigenvalue $k$. Apply this operator $j$ times and then set $a=1/2$.
The relation $y(a)=0$ gives $S_1$, and differentiation of
$y(z)=z(1-4z^2)/G(4z^2)$, where
$G(s)=\sum_L\Delta_Ls^L$, gives the remaining formulas using
$G(1)=B$, $G'(1)=M$, and $G''(1)=\mu_2-M$.

Since $w_1=1/12$, $w_2=-b_1/20$, and
$\dot s_k/p_k=-\dot w_k$, the two exceptional terms in
\eqref{eq:quartic-exceptions} contribute
$A^2/6-A^2b_1/2$ to $U_3$ and $-A^2\dot b_1/2$ to the forcing
sum before multiplication by $2/A$. Put $\alpha=\alpha_3$.
Substitution in \eqref{eq:forcing} expresses
$\mathfrak F_3(t)\,\dd t$ as
\begin{align}
 \frac2A\left\{
 \sum_{j=2}^4q_j\left(\frac{r_j}{A}\,\dd A-\dd r_j\right)
 -\left(\alpha^2-\frac\alpha2\right)\dd A
 +\alpha\left(2AB\,\dd S_2+4AM\,\dd S_1-2A\,\dd A\right)
 \right.\notag\\[-2mm]
 \left.\hspace{25mm}
 +\left(\frac A6-\frac{Ab_1}{2}\right)\dd A
 -\frac{A^2}{2}\,\dd b_1\right\}.
 \label{eq:finite-forcing}
\end{align}
The affine terms vanish separately, since
\[
 \frac{S_j}{A}\,\dd A-\dd S_j=0\quad(j=0,1),
 \qquad S_0=A,\quad S_1=-A/2.
\]
For completeness, the differentials needed in the remaining terms are
\begin{align*}
 \dd r_2={}&\frac{\dd B}{8B^2},\\
 \dd r_3={}&\left(\frac5{16B^2}-\frac{M}{2B^3}\right)\dd B
                 +\frac{\dd M}{4B^2},\\
 \dd r_4={}&\left(\frac{23}{32B^2}-\frac{2M}{B^3}
       -\frac{3\mu_2}{4B^3}+\frac{9M^2}{4B^4}\right)\dd B
       +\left(\frac1{B^2}-\frac{3M}{2B^3}\right)\dd M
       +\frac{3\,\dd\mu_2}{8B^2}.
\end{align*}
The coefficient of $\dd A$ in \eqref{eq:finite-forcing} reduces to
\[
 \frac2{A^2}\sum_{j=2}^4q_jr_j+\frac{\alpha}{2A}+\frac13-b_1,
\]
because $A(B/2-2M-2)=\alpha-1/4$.
Direct multiplication of \eqref{eq:q234} and \eqref{eq:Snext} gives
\[
 \sum_{j=2}^4q_jr_j
 =A^2\left(-\frac34\mu_2+\frac23M+\frac7{48}B
             -\frac{M^2}{6B}\right)-\frac{A\alpha}{4}.
\]
For $\dd B$ and $\dd M$, the corresponding contractions are
\begin{align*}
 -\sum_{j=2}^4q_j\partial_Br_j+\frac{\alpha A}{4B}
 &=A^2\left(\frac7{48}+\frac{M^2}{6B^2}\right),\\
 -\sum_{j=2}^4q_j\partial_Mr_j
 &=A^2\left(\frac23-\frac{M}{3B}\right).
\end{align*}
The coefficient of $\dd\mu_2$ is
$-(2/A)q_4\,3/(8B^2)=-3A/2$, and that of $\dd b_1$ is $-A$.
These five computations give

\begin{align}
 &\left(\frac13-b_1-\frac32\mu_2+\frac43M+\frac7{24}B
 -\frac{M^2}{3B}\right)\dd A\notag\\
 &\quad+A\left[-\dd b_1-\frac32\,\dd\mu_2
 +\left(\frac43-\frac{2M}{3B}\right)\dd M
 +\left(\frac7{24}+\frac{M^2}{3B^2}\right)\dd B\right].
 \label{eq:closed-oneform}
\end{align}
This is the differential of the expression in braces in
\eqref{eq:primitive}. The calculation is finite once the moment sums
have been taken; all such contractions were justified above.

At $t=0$, the primitive in \eqref{eq:primitive} has value $-1/80$,
which is $\beta_3^{(0)}$. It therefore equals $\mathfrak b_3(t)$.
For the family considered here,
\begin{equation}\label{eq:mu2}
 b_1=1+t(\pi^2/6-1),\quad M=\tfrac34tB,\quad
 \mu_2=B\left[t\left(\frac{\pi^2}{12}+\frac1{16}\right)
 +\frac9{16}t^2\right].
\end{equation}
The last formula follows by logarithmic differentiation of $H$ and
the identity
$\sum_{j\geq2}j^2/(j^2-1)^2=\pi^2/12+1/16$.
Substitution in \eqref{eq:b2-path} gives
\begin{equation}\label{eq:b2-closed}
 \begin{aligned}
 \mathfrak b_2(t)=A_t\left[-\frac23+
 \left(1-\frac{\pi^2}{6}\right)t
 +2^t\left\{-\frac{29}{24}
 +\left(\frac{77}{32}-\frac{\pi^2}{8}\right)t
 -\frac{33}{32}t^2\right\}\right].
 \end{aligned}
\end{equation}
This proves the two-marking case of
Theorem~\ref{thm:first-volume-correction}:
\begin{equation}\label{eq:explicit-refinement}
 V^{(t)}_{g,2}=
 \frac{\Gamma(2g-1/2)}{\sqrt2\pi^{3/2}
 A_t^{2g-1/2}\sqrt{\kappa_t}}
 \left(1+\frac{\mathfrak b_2(t)}g+O(g^{-2})\right),
 \qquad 0\leq t\leq1.
\end{equation}
In particular, $\mathfrak b_2(0)=-1/16$ and
$\mathfrak b_2(1)=-5/48+1/(6\pi^2)$, in agreement with the endpoint
coefficients. For each fixed $n\geq2$, summation of the first marking
correction gives
\begin{equation}\label{eq:bn-closed}
 \mathfrak b_n(t)=\mathfrak b_2(t)
 +A_t(n-2)\left[\frac{B_t(6-n)}2+3-n-2M_t\right].
\end{equation}
The corresponding expansion has the leading normalization in
\eqref{eq:all-n-asym} and relative remainder $O_n(g^{-2})$.
This completes the proof of Theorem~\ref{thm:first-volume-correction}.
This assertion about the first correction is proved only for
$n\geq2$. The leading asymptotic for all marking counts does not
imply the corresponding correction formula for $n=0$ or $n=1$.

\section{Conventions in the cited formulas}
We collect the conventions relating the formulas used above to their
sources.

Eynard's intersection formula
\cite[Theorem 3.2, (3.10)--(3.12)]{Eyn} applies to the deformed
Airy curve. The change of kernel $K=-2K_E$ cancels the factor
$(-2)^{2-2g-n}$ in that formula. The first forgetful-map identities
appear in (3.6)--(3.7) of the same paper. Grouping the general
permutation formula by cycles gives the logarithmic Schur transform
used in \eqref{eq:shifted}.

The Laurent residue calculation in Section~\ref{sec:exact} gives
the coefficient recursion \eqref{eq:rec} directly. In particular,
it determines the factors $4$ and $8$ and the upper limit
$L\leq3g-3+n-|\mathbf d|$ for the whole family.

The integrated KdV coefficient identity in
\cite[Section 5, proof of Proposition 5.1]{LX} has coefficients
$1/12$ and $1/2$ on its two right-hand terms. Multiplication by
$p_1=12$ gives the coefficients $1$ and $6$ in \eqref{eq:kdv}.

The sine-curve calculation in
\cite[Section 8.1, (8.8)--(8.10)]{Eyn} identifies the recursion
with the Laplace transform of
$\langle\exp(2\pi^2\kappa_1+\tfrac12\sum L_i^2\psi_i)\rangle$.
The same change of kernel removes the prefactor in that calculation.
The resulting stack normalization is $V_{1,1}=\pi^2/12$.

For Painlev\'e~I, the coefficient asymptotic used in
Section~\ref{sec:anchor} is
\cite[Section 2.3, (2.39)]{GIKM}. Its expansion to all orders is
given in (2.43)--(2.45) of that paper. These are the results for the
zero-instanton coefficients. The conjectures concerning the higher
instanton sectors are not used here.

\section{Examples and finite calculations}
For the cubic specialization, the intersection formula and the
recurrence \eqref{eq:c-rec} give
\[
 \begin{array}{c|r}
 g&V^{(0)}_{g,2}\\\hline
 1&12\\
 2&42336\\
 3&609638400\\
 4&19763038181376
 \end{array}
\]
At three markings one has
\[
 \frac{J^{(0)}_{g,3}(2,0,0)}{V^{(0)}_{g,3}}
 =\frac{60g-20}{60g-12}
 =1-\frac{2}{15g}+O(g^{-2}),
\]
in agreement with $E_{3,2}(0)=2/15$.

Consider also the rational kernel
\[
 q(s)=\frac1{(1-s)(1-s/4)},\qquad
 y(z)=z-5z^3+4z^5,\qquad s_2=15,\quad s_3=-60.
\]
The negative value of $s_3$ illustrates that the shifted times
need not be positive even when the kernel coefficients are positive.

For $q(s)=(1-s)^{-1}(1-s/4)^{-1/2}$, the Taylor coefficients and
shifted times are rational. The first times are
$s_2=27/2$, $s_3=-225/8$, and $s_4=-735/16$.
Here $B^2=4/3$, so homogeneous cubic comparisons can be expressed
in rational arithmetic by squaring positive quotients.

The accompanying verification scripts compute intersection numbers by
the DVV recursion and then form the shifted specialization. For each
of the cubic and rational polynomial curves above, they check
29 instances of \eqref{eq:rec} and 8 instances of \eqref{eq:kdv}
in exact rational arithmetic. Additional scripts check the quartic
formula, its cutoff corrections, and the five coefficients of the
differential in \eqref{eq:closed-oneform}. The scripts specify all
test inputs. These calculations check finite identities and
normalizations; the uniform estimates follow from the proofs above.

Two restrictions are relevant in applying the formulas. The stable
genus-zero factors in KdV must be retained, since they can contribute
at order $g^{-1}$. The first descendant correction requires
$n\geq3$, so that each merge term has at least two markings.
The variation of $V_{g,2}$ inserts a third marking and hence lies
within this range.

\end{document}